\documentclass{amsart}

\usepackage[textsize=tiny,textwidth=4cm]{todonotes}
\usepackage{amssymb}
\usepackage{amsfonts}
\usepackage{mathtools}
\usepackage{array}
\usepackage{capt-of}
\usepackage[skip=2pt,font=footnotesize]{caption}
\usepackage[shortcuts]{extdash}
\usepackage{adjustbox}
\usepackage{pgfplots}
\pgfplotsset{compat=1.12}

\usepackage{hyperref}
\hypersetup{bookmarks,
	raiselinks,
	colorlinks,
	citecolor=black,
	linkcolor=black,
	urlcolor=black,
	filecolor=black,
	menucolor=black}
\usepackage[capitalize,nameinlink]{cleveref}
\crefname{section}{section}{sections}
\crefname{subsection}{subsection}{subsections}
\Crefname{section}{Section}{Sections}
\Crefname{subsection}{Subsection}{Subsections}
\crefformat{equation}{\textup{#2(#1)#3}}
\crefrangeformat{equation}{\textup{#3(#1)#4--#5(#2)#6}}
\crefmultiformat{equation}{\textup{#2(#1)#3}}{ and \textup{#2(#1)#3}}
{, \textup{#2(#1)#3}}{, and \textup{#2(#1)#3}}
\crefrangemultiformat{equation}{\textup{#3(#1)#4--#5(#2)#6}}%
{ and \textup{#3(#1)#4--#5(#2)#6}}{, \textup{#3(#1)#4--#5(#2)#6}}{, and \textup{#3(#1)#4--#5(#2)#6}}

\Crefformat{equation}{#2Equation~\textup{(#1)}#3}
\Crefrangeformat{equation}{Equations~\textup{#3(#1)#4--#5(#2)#6}}
\Crefmultiformat{equation}{Equations~\textup{#2(#1)#3}}{ and \textup{#2(#1)#3}}
{, \textup{#2(#1)#3}}{, and \textup{#2(#1)#3}}
\Crefrangemultiformat{equation}{Equations~\textup{#3(#1)#4--#5(#2)#6}}%
{ and \textup{#3(#1)#4--#5(#2)#6}}{, \textup{#3(#1)#4--#5(#2)#6}}{, and \textup{#3(#1)#4--#5(#2)#6}}

\makeatletter
\DeclareRobustCommand*{\bbl@ap}[1]{\textormath{\textsuperscript{#1}}{^{\mathrm{#1}}}}%
\DeclareRobustCommand*{\bbl@ped}[1]{\textormath{$_{\mbox{\fontsize\sf@size\z@ \selectfont#1}}$}{_\mathrm{#1}}}%
\let\ap\bbl@ap
\let\ped\bbl@ped
\makeatother

\newtheorem{theorem}{Theorem}[section]

\newtheorem{proposition}[theorem]{Proposition}
\newtheorem{lemma}[theorem]{Lemma}
\newtheorem{assumption}[theorem]{Assumption}
\newtheorem{remark}[theorem]{Remark}

\newtheorem{example}[theorem]{Example}

\newcommand*{\tcolvec}[1]{\begin{psmallmatrix}#1\end{psmallmatrix}}
\newcommand{\Vhpr}{{{V_h}^{\hspace{-0.25em}\prime}}}

\newcommand{\dd}{\,\mathrm{d}}

\newcommand{\supp}{\operatorname{supp}}

\newcommand{\spanlin}{\operatorname{span}} 
\newcommand{\clos}{\operatorname{clos}}  
\newcommand{\la}{\langle} 
\newcommand{\ra}{\rangle}
\newcommand{\R}{\mathbb{R}} 
\newcommand{\N}{\mathbb{N}}

\newcommand{\cX}{{\mathcal{X}}}
\newcommand{\cXd}{{\mathcal{X}_{\delta}}}

\newcommand{\cY}{{\mathcal{Y}}}
\newcommand{\cYd}{{\mathcal{Y}_{\mkern-2mu\delta}}}

\ifpdf
\hypersetup{
  pdftitle={Stable and efficient Petrov-Galerkin methods for a kinetic Fokker-Planck equation},
  pdfauthor={J. Brunken and K. Smetana}
}
\fi

\title[Petrov-Galerkin methods for a kinetic Fokker-Planck equation]{Stable and efficient Petrov-Galerkin methods for a kinetic Fokker-Planck equation}

\author{Julia Brunken}
\address{University of M{\"u}nster, Applied Mathematics, 
	Einsteinstr.\ 62, 48149 M{\"u}nster, Germany, 
	{julia.brunken@uni-muenster.de}}

\author{Kathrin Smetana}
\address{University of Twente, Faculty of Electrical Engineering, Mathematics \& Computer Science,
	Zilverling,  P.O. Box 217, 7500 AE Enschede, The Netherlands. Current address: Department of Mathematical Sciences, Stevens Institute of Technology, 1 Castle Point Terrace, Hoboken, NJ 07030, United States of America, {ksmetana@stevens.edu}.
}
\date{\today}

\thanks{The work of Julia Brunken was supported by the German Federal Ministry of Education and Research under grant BMBF 05M2016 - GlioMaTh and by the Deutsche Forschungsgemeinschaft (DFG, German Research Foundation) under Germany's Excellence Strategy EXC 2044 –390685587, Mathematics Münster: Dynamics–Geometry–Structure.}



\begin{document}
	
	\begin{abstract}
		We propose a stable Petrov-Galerkin discretization of a kinetic Fokker-Planck equation constructed in such a way that uniform inf-sup stability can be inferred directly from the variational formulation. 
		Inspired by well-posedness results for parabolic equations, we derive a lower bound for the dual inf-sup constant of the Fokker-Planck bilinear form by means of stable pairs of trial and test functions.
		The trial function of such a pair is constructed by applying the kinetic transport operator and the inverse velocity Laplace-Beltrami operator to a given test function. 
		For the Petrov-Galerkin projection we choose an arbitrary discrete test space and then define the discrete trial space using the same application of transport and inverse Laplace-Beltrami operator. 
		As a result, the spaces replicate the stable pairs of the continuous level and we obtain a well-posed numerical method with a discrete inf-sup constant identical to the inf-sup constant of the continuous problem independently of the mesh size.
		We show how the specific basis functions can be efficiently computed by low-dimensional elliptic problems, and confirm the practicability and performance of the method {with numerical experiments}.
	\end{abstract}
	
	\maketitle

\section{Introduction}

In this manuscript we develop a stable and efficient Petrov-Galerkin approximation scheme for {certain} kinetic Fokker-Planck equations{, including the equation}
\begin{equation} \label{eq:fp_intro}
\partial_t u((t,x),v) + v \cdot \nabla_x u((t,x),v) = \Delta_{v} \left (\tfrac{u((t,x),v)}{q(x,v)} \right ) \quad \text{in } \Omega = I_t \times \Omega_{x} \times \Omega_v \\
\end{equation}
with suitable inflow boundary conditions. \Cref{eq:fp_intro} describes a particle density $u$ dependent on time $t \in I_t$, position $x \in \Omega_x \subset \R^{d}$, $d\in \{2,3\}$, and { direction} $v \in \Omega_v = S^{d-1}$, {where $S^{d-1}$ is the $(d-1)$-dimensional unit sphere and $q \in L^\infty(\Omega_{x} \times \Omega_v)$ with $q > 0$ a.e.\ and $q(x,\cdot) \in C^1(\Omega_v)$ for a.e.\ $x \in \Omega_{x}$.}

Formulations for particle densities governed by kinetic equations arise in various contexts. Beyond the classical applications of radiative transfer and kinetic gas theory (see e.g.\ \cite{DavSyk1957,DudMar1979}), kinetic equations are, for instance, also used to describe densities of tumor cells in multiscale descriptions of tumor spreading \cite{EHKS15,Hunt2017}. { In this manuscript, we are mainly interested in the  latter application. More precisely, we focus on a discretization of a prototype of a glioma tumor equation described in \cite{Hunt2017},  where the velocity is driven by a Brownian motion resulting in the specific Laplace-Beltrami term of \eqref{eq:fp_intro}. However, other variants including, e.g., $\nabla_v$ terms, are also included in the more general setting considered in the course of this work.} 

We aim for a finite element discretization with guaranteed stability. Therefore, we focus on a Petrov-Galerkin discretization based on a stable variational formulation of \eqref{eq:fp_intro}, since in such a framework the well-posedness of the discrete scheme can often be inferred from respective results on the continuous level, see e.g.\ \cite{DHSW2012,DemGop11,SS2009,UP2014}. 

First, we establish a full-dimensional variational formulation for \eqref{eq:fp_intro} based on Bochner-type spaces{,} mapping the combined space-time domain $\Omega_{t,x} = I_t \times \Omega_{x}$ to a Sobolev space defined on the velocity domain $\Omega_v$ similar to spaces defined in \cite{AM2019,Car1998}.
Taking the viewpoint that the Fokker-Planck equation {can} be interpreted as a ``generalization'' of a parabolic equation with a $(d+1)$-dimensional kinetic transport operator $\partial_t + v \cdot \nabla_x$ instead of a one-dimensional time derivative $\partial_t$, we analyze the well-posedness of the variational formulation for \eqref{eq:fp_intro} by combining respective approaches developed for parabolic equations \cite{EG2004,SS2009, UP2014} and for transport equations \cite{BSU2019,DHSW2012,DemGop11}. 
We show existence of a weak solution by verifying the dual inf-sup condition{.} To that end, similarly to \cite{EG2004,SS2009}{,} specific function pairs in the trial and test spaces are constructed{.} We associate a test space function $p$ to a trial space function roughly defined as $w_p = p - (\Delta_v)^{-1} (\partial_t p + v \cdot \nabla_x p)$. Then the bilinear form evaluated in $w_p$ and $p$ can be bounded from below by the respective norms of $w_p$ and $p$, which leads to a lower bound for the dual inf-sup constant. This approach is a generalization of proofs for parabolic equations using a variant of $w_p$ containing only the time derivative instead of the kinetic transport operator \cite{EG2004,SS2009} and of proofs for transport equations, where a ``stable function pair'' consists roughly of { $-(\partial_t p + v \cdot \nabla_x p)$ and $p$, when choosing the kinetic transport operator in the linear transport equation}, see \cite{BSU2019,DHSW2012,DemGop11}. 
Under an additional assumption on the global traces of certain functions, we also show uniqueness of the solution similar to proofs for parabolic equations \cite{EG2004} and transport equations \cite{Azerad1996}, and have a stability estimate dependent on the inf-sup constant{,} which is similar to the respective estimates for parabolic equations.  

To design the Petrov-Galerkin discretization, we use problem-specific trial spaces ensuring stability:  
We first choose an arbitrary discrete test space  $\cYd$ and then define the discrete trial space roughly as $\cXd = \cYd + (\Delta_v)^{-1} (\partial_t + v \cdot \nabla_x )\cYd$. 
The spaces thus consist of pairs $w_p^\delta, p^\delta$ that are the discrete counterparts of the pairs $w_p,p$ used in the proof for the lower bound of the dual inf-sup constant.  
This approach automatically yields a well-posed discrete problem with the same stability constant as for the continuous problem independently of the choice of the test space and thus of the mesh size.
The strategy to use an application of the transport operator for defining a stable trial space was already used for linear first-order transport equations \cite{BSU2019} and for the wave equation \cite{HPSU2019} as an alternative to computing stable test spaces by approximately inverting the transport operator \cite{DHSW2012,DemGop11}. 
Our choice ensures that the spaces can be efficiently computed in the course of the numerical scheme, where we apply the high-dimensional transport operator and only solve low-dimensional elliptic problems in the velocity domain due to the inverse Laplace-Beltrami operator. 
As a result, we can guarantee the stability of the method 
with low-dimensional computations that are not dominant in the computational costs of the full solution process.

Weak solutions and variational formulations for different types of kinetic Fokker-Planck equations have been defined and analyzed in various works, see e.g.\ \cite{AM2019,Car1998,DegMas1987,HJJ18,SheHan13}. However, these approaches focus on the properties of the weak solution without an orientation towards a subsequent discretization.
On the other hand, discretizations of kinetic Fokker-Planck equations are often not based on the direct connection to a weak solution or do not specifically consider stability estimates. 
In \cite{LehtikangasEtAl10}, a finite element discretization of a kinetic Fokker-Planck equation is described, where the well-posedness of the discrete problem is however not analyzed. 
Applying the framework of \cite{EggSchl12}, a mixed variational formulation  with a subsequent discretization for a generalized Fokker-Planck equation is proposed in \cite{HLST13}.
In the context of neuronal networks, a Fokker-Planck equation is discretized with finite differences in \cite{CCT2011}.
Another well-established approach to discretize kinetic equations is the method of moments, applied to Fokker-Planck equations, for instance, in \cite{FHK06,SAFK14}, while a related approach in the context of hierarchical model reduction is proposed in \cite{BLOS16}.
For the related Vlasov-Fokker-Planck system there are, for instance, works based on finite differences \cite{Schaeffer1998, WolOzi2009} and streamline-diffusion discontinuous Galerkin approximations \cite{AsaKow2005,AsaSop2007}. For the more general class of equations with nonnegative characteristic form, discontinuous Galerkin methods \cite{HSS2002,HS2001} and also sparse tensor approximations \cite{SST2008} have been developed. 

This paper is structured as follows. After a more detailed description of the considered Fokker-Planck equation in \cref{sect:FPeq}, we introduce the suitable Bochner-type function spaces and establish density and trace properties in \cref{sect:function_spaces}. We then derive the variational formulation and prove the existence and uniqueness results in \cref{sect:var_form}. In \cref{sect:discretization}, we introduce the discrete scheme, show well-posedness and describe an efficient computation. These properties of the proposed method are finally confirmed for a numerical example in \cref{sect:num_exp}.

\section{The kinetic Fokker-Planck equation} \label{sect:FPeq}

In this paper we consider a simplified version of the kinetic Fokker-Planck equation developed in \cite[sect. 2.4.2]{Hunt2017} that gives a mesoscopic description of the density of glioma tumor cells. 
Let $\Omega_x \subset \R^d$, $d \in \{2,3\}$ be the spatial domain\footnote{One can also define a Fokker-Planck equation on a one-dimensional spatial domain, where the velocity has to be defined as a one-dimensional projection variable, see, e.g., \cite{SAFK14}. We leave out this special case for ease of presentation.} with piecewise $C^1$ boundary that is globally Lipschitz and let $I_t := (0,T)$ be the time interval. Moreover, let the velocity domain be the $(d-1)$-dimensional unit sphere $\Omega_v := S^{d-1}$, which corresponds to the assumption of particles with constant speed but varying direction. As we will often treat space and time variables simultaneously, we denote by $\Omega_{t,x} := I_t \times \Omega_x$ the space-time domain. The full domain is defined as $\Omega := \Omega_{t,x} \times \Omega_v$.  

{To prescribe suitable inflow boundary conditions, we first define relevant boundaries. First, we denote by
	\begin{equation*}
	\hat \Gamma := \{0, T\} \times \bar \Omega_x \times \Omega_v \cup [0,T] \times \partial \Omega_x \times \Omega_v
	\end{equation*}
	the essential boundary of $\Omega$.}
Then, we define the spatial {out- and inflow} domains
$
\Gamma_{\pm}^x(v) := \{ x \in \partial \Omega_x : n(x) \cdot v \gtrless 0 \} \subset \partial \Omega_x,
$
where $n(x)$ is the unit outer normal to $\partial \Omega_x$ at $x$. The full {out- and inflow} domains {$\Gamma_+$ and $\Gamma_-$} are then defined as
\begin{equation*}
\Gamma_{\pm} := \left \{ 
((t,x),v) \in \partial \Omega_{t,x} \times \Omega_v \,:\, 
\tcolvec{ 1 \\ v  } \cdot n(t,x) \gtrless 0
\right \} \subset {\hat\Gamma},
\end{equation*}
where $n(t,x)$ is the unit outer normal to $\partial \Omega_{t,x}$ at $(t,x)$.
{The sets} $\Gamma_{\pm}$ thus contain both the temporal and the spatial boundaries, i.e., $\Gamma_-$ contains the ``initial boundary'' and the ($v$-dependent) spatial inflow boundary whereas $\Gamma_+$ contains the final time boundary and the spatial outflow boundary. 

The strong form of the Fokker-Planck equation then reads
\begin{equation} \label{strong_form_with_q}
\begin{aligned}
\partial_t u((t,x),v) + v \cdot \nabla_x u((t,x),v) &= \Delta_{v} \left (\tfrac{u((t,x),v)}{q(x,v)} \right )  &&\text{in } \Omega,\\
u((t,x),v) &= g((t,x),v) &&\text{on } \Gamma_-,
\end{aligned}
\end{equation}
where $\Delta_{v}$ is the Laplace-Beltrami operator on the unit sphere $\Omega_v = S^{d-1}$, {$q \in L^\infty(\Omega_{x} \times \Omega_v)$ is the so-called ``tissue fiber orientation distribution'' satisfying $q(x,\cdot) \in C^1(\Omega_v)$ for a.e.\ $x \in \Omega_{x}$ and $q \geq \alpha_q > 0$ a.e.\ in $\Omega_{x} \times \Omega_v$} and $g : \Gamma_- \to \R$ is the inflow boundary condition that contains the initial condition $g|_{\{t = 0\}}$ as well as the spatial inflow boundary condition $g|_{\Gamma_-^x(v)}, v \in \Omega_v$. { Since $q$ is assumed to be sufficiently regular, we can bring the respective differential operator in \cref{strong_form_with_q} in divergence form.´}

In \cref{sect:var_form}, we develop a variational formulation for this equation, where we allow for a more general differential operator on $\Omega_v$ and give specific conditions on $q$ and $g$ leading to well-posedness.

\section{Function spaces} \label{sect:function_spaces}

To develop a variational formulation for \cref{strong_form_with_q} we first introduce the necessary function spaces. Since we aim for a {full space-time-velocity formulation}, we use Bochner spaces mapping the space-time domain $\Omega_{t,x}$ to a space of functions on $\Omega_v$.

We start with the function space for the velocity variable: Since the equation contains a Laplace-Beltrami operator on the velocity domain $\Omega_v = S^{d-1}$, we define $V := H^1(\Omega_v) \subset L^2(\Omega_v)$ as the Sobolev space of weakly differentiable functions on the surface $\Omega_v = S^{d-1}$ with {squared} norm $\|\phi\|_{V}^2 = \|\phi\|_{L^2(\Omega_v)}^2+\|\nabla_v \phi\|_{L^2(\Omega_v)}^2$. For details on the definition of Sobolev spaces on manifolds, see \cite{DE2013, Hebey2000}.
We denote the dual space of $V$ by $V' := H^{-1}(\Omega_v)$.
{The space} $V$ is a dense subspace of $L^2(\Omega_v)$ and we will make use of the Gelfand triple $V \hookrightarrow L^2(\Omega_v) \hookrightarrow V'$, where we denote the dual pairing by $\la \cdot, \cdot \ra_{V'\!,V}$.

As {a} function space for the full domain, we will use the space $L^2(\Omega_{t,x}; V)$ with {squared} norm
\begin{equation}\label{eq:L2OmegaxtV}
\|w\|_{L^2(\Omega_{t,x};V)}^2 = \int_{\Omega_{t,x}} \|w(t,x)\|_{V}^2 \dd (t,x).
\end{equation}
{
	From now on, we will denote the kinetic advection field $\tcolvec{1\\v}$ by $k \in C^1(\bar\Omega,\R^{d+1})$, $k((t,x),v) := \tcolvec{1\\v}$, so that the kinetic space-time transport operator is given as $k \cdot \nabla_{t,x} p  = \partial_t p + v \cdot \nabla_x p$. 
	We then define }
\begin{equation} \label{eq:def_h1kin}
H^1_{\mathrm{FP}}(\Omega) := \{ p \in L^2(\Omega_{t,x}; V) \,:\, {k} \cdot \nabla_{t,x} p \in L^2(\Omega_{t,x};V') \},
\end{equation}
with {squared} norm
\begin{equation} \label{eq:h1kin_norm}
\|p\|_{H^1_{\mathrm{FP}}(\Omega)}^2 := 
\|p\|_{L^2(\Omega_{t,x}; V)}^2 + \|{k} \cdot \nabla_{t,x} p\|_{L^2(\Omega_{t,x};V')}^2.
\end{equation} 
This definition is similar to the spaces used for other variants of the kinetic Fokker-Planck equation{,} e.g.{,} in \cite{AM2019,BalPal2020,Car1998}. We use ideas from \cite{AM2019} to show the following:
\begin{proposition} \label{prop:density}
	The set $C^\infty(\bar\Omega_{t,x} \times \Omega_v)$ is dense in $H^1_{\mathrm{FP}}(\Omega)$.
\end{proposition}
\begin{proof}
	{
		For the proof one constructs approximations of a function $f \in H^1_{\mathrm{FP}}(\Omega)$ by a mollification in $\Omega_{t,x}$ analogously to  \cite[Prop.\ 7.1]{AM2019} and a suitable basis expansion in $\Omega_v$. For more details see the supplementary material.
	}
\end{proof}

To discuss the boundary behavior of functions in $H^1_{\mathrm{FP}}(\Omega)$, we introduce weighted $L^2$-spaces{,} as usually used for transport and kinetic equations (e.g. \cite{Bar70,Cessenat1984}, \cite[XXI, \S 2]{DauLio1993_Vol6}) and for different versions of the kinetic Fokker-Planck equation \cite{AM2019,Car1998}. For any $\Gamma \subseteq {\hat\Gamma}$ we introduce $L^2(\Gamma,$ $ |{k} \cdot n|)$ with {squared} norm
\begin{equation}
\|w\|^2_{L^2(\Gamma, |{k} \cdot n|)} := \int_{\Gamma} w^2 \left | {k} \cdot n \right | \dd s.  
\end{equation}
Then, we can show that functions in $H^1_{\mathrm{FP}}(\Omega)$ admit local traces  {{on $\Gamma_+ \cup \Gamma_-$}:
	\begin{proposition} \label{lem:trace}
		For every compact set $K \subset \Gamma_+$ (resp. $K \subset \Gamma_-$), the trace operator $w \mapsto w|_{K}$ from $C^\infty(\bar \Omega)$ to $L^2(K, |{k} \cdot n|)$ extends to a continuous linear operator on $H^1_{\mathrm{FP}}(\Omega)$.
	\end{proposition}
	For the proof we need to estimate the product of $H^1_{\mathrm{FP}}(\Omega)$ functions with different test functions in the following way, where the proof can be found in \cref{sect:app_proofs}.
	\begin{lemma} \label{lem:product_continuous}
		Let $\phi \in C^1(\bar \Omega)$. Then, the mapping $f \mapsto \phi f$ is continuous in $H^1_{\mathrm{FP}}(\Omega)$ with the estimate
		\begin{equation*}
		\|\phi f\|_{H^1_{\mathrm{FP}}(\Omega)}
		\leq 
		C \|\phi\|_{C^1(\Omega)} \|f\|_{H^1_{\mathrm{FP}}(\Omega)}.
		\end{equation*}
	\end{lemma}
	
	\begin{proof}[Proof of \cref{lem:trace}]
		We use ideas of the proof of a similar result for transport equations{,} e.g.{,}  in \cite[Chap.\ XXI, Thm.\ 1, p.\ 220]{DauLio1993_Vol6}. Analogous results for spaces similar to $H^1_{\mathrm{FP}}(\Omega)$ are also given in \cite[Proofs of Lemmas 4.3, 7.6]{AM2019}.
		
		Given a compact set $K \subset \Gamma_+$, let $\eta_K \in C^1(\bar\Omega)$ with $\eta_K = 1$ on $K$ and $\supp \eta_K \cap \Gamma_- = \emptyset$.  
		We then obtain by integrating by parts for $w \in C^\infty(\bar \Omega)$
		\begin{align*}
		\int_K w^2 |{k} \cdot n| \dd s
		&= \int_K (\eta_K w)^2 |{k} \cdot n| \dd s
		\leq \int_{{\hat\Gamma}} (\eta_K w)^2 |{k} \cdot n |  \dd s \\
		& \overset{(*)}{=} \int_{{\hat\Gamma}} (\eta_K w)^2 {k} \cdot n  \dd s 
		= 2 \int_{\Omega} \eta_K w {k} \cdot \nabla_{t,x} (\eta_K w) \dd((t,x),v) \\
		&
		\leq 2\|\eta_K w\|_{L^2(\Omega_{t,x},V)} \|{k} \cdot \nabla_{t,x} (\eta_K w)\|_{L^2(\Omega_{t,x},V')} \\
		&
		\leq 2\|\eta_K w\|_{H^1_{\mathrm{FP}}(\Omega)}^2 
		\!\!\overset{\text{\cref{lem:product_continuous}}}{\leq}\!\!
		C \|\eta_K\|_{C^1(\Omega)}^2 \|w \|_{H^1_{\mathrm{FP}}(\Omega)}^2.
		\end{align*}
		We thus have continuity of the mapping $w \mapsto w|_{K}$ for all $w \in C^\infty(\bar \Omega)$, and by density (\cref{prop:density}) the mapping extends to a continuous operator $H^1_{\mathrm{FP}}(\Omega) \to L^2(K,|{k} \cdot n |)$. For $K \subset \Gamma_-$ the claim can be shown analogously using $|{k} \cdot n | = - {k} \cdot n$ on $\supp \eta_K$ in $(*)$.
	\end{proof}
	This result ensures that $H^1_{\mathrm{FP}}(\Omega)$ functions have a trace on the non\-/characteristic boundary\footnote{The non-characteristic boundary is the part of the boundary where $|{k} \cdot n | \neq 0$.} $\Gamma_+ \cup \Gamma_-$. 
	However, from the local existence of traces we cannot directly deduce that these generally lie in global trace spaces as e.g.\ $L^2(\partial\Omega, |{k} \cdot n|)$.
	
	{We now define}
	\begin{equation} \label{def:HkinG+-}
	H^1_{\mathrm{FP},{\Gamma_\pm}}(\Omega) := \clos_{\|\cdot\|_{H^1_{\mathrm{FP}}(\Omega)}} \{f \in C^\infty(\bar \Omega) \,:\, f \equiv 0 \text{ on } {\Gamma_\pm}\}.
	\end{equation}
	{ To avoid boundary integrals on the outflow domain in the variational formulation, we will use $H^1_{\mathrm{FP},{\Gamma_+}}(\Omega)$ } as the test space for our variational formulation.} 
With the restriction of functions in $H^1_{\mathrm{FP},\Gamma_+}(\Omega)$ on the outflow boundary and the definition through the closure, we can show that these functions have a trace in $L^2(\Gamma_-, |{k} \cdot n|)$:
\begin{proposition} \label{prop:trace_PI} 
	There exists a linear continuous mapping
	$
	\gamma_- : H^1_{\mathrm{FP},\Gamma_+}(\Omega) \to L^2(\Gamma_-, |{k} \cdot n|)
	$
	such that
	\begin{equation*}
	\|\gamma_-(w)\|_{L^2(\Gamma_-, |{k} \cdot n|)} \leq C \|w\|_{H^1_{\mathrm{FP}}(\Omega)} \quad \forall w \in H^1_{\mathrm{FP},\Gamma_+}(\Omega).
	\end{equation*}
	Furthermore, the integration by parts formula
	\begin{equation*}
	\int_{\Omega_{t,x}} \la {k} \cdot \nabla_{t,x} w , w \ra_{V'\!, V} \dd (t,x) = \tfrac{1}{2} \int_{\Gamma_-} w^2   {k} \cdot n  \dd s 
	\end{equation*}
	holds for all $w \in H^1_{\mathrm{FP},\Gamma_+}(\Omega)$.
\end{proposition}
\begin{proof}
	The proof is similar to the respective result for transport equations e.g. in \cite[Prop.\ 2.4]{BSU2019}, see also \cite[sect.\ 4]{AM2019}. Let $w \in C^\infty(\bar\Omega)$ with $w \equiv 0$ on $\Gamma_+$. Performing integration by parts we obtain
	\begin{align*}
	\int_{\Omega}  w {k} \cdot \nabla_{t,x} w \dd((t,x),v)
	&= - \int_{\Omega} \nabla_{t,x} w \cdot {k} w \dd ((t,x),v)
	+ \int_{\Gamma_-} w^2  \underbrace{{k} \cdot n}_{<0}  \dd s,
	\end{align*}
	and thus
	\begin{align*}
	\|w\|_{L^2(\Gamma_-, |{k} \cdot n|)}^2 
	&= \int_{\Gamma_-} w^2  \left | {k} \cdot n \right | \dd s 
	= 2  \int_{\Omega} (- {k} \cdot \nabla_{t,x} w) w \dd ((t,x),v)  \\
	& \leq 2 \|- {k} \cdot \nabla_{t,x} w \|_{L^2(\Omega_{t,x}; V')}  \|w\|_{L^2(\Omega_{t,x};V)}
	\leq 2 \|w\|_{H^1_{\mathrm{FP}}(\Omega)}^2.
	\end{align*}
	By density (due to the definition of $H^1_{\mathrm{FP},\Gamma_+}(\Omega)$){,} the integration by parts formula and the bound for $\|w\|_{L^2(\Gamma_-,|{k} \cdot n|)}$ hold for all $w \in H^1_{\mathrm{FP},\Gamma_+}(\Omega)$.
\end{proof}

\begin{remark} \label{remark:H1kinGamma-}
	Similarly, it can be shown that the space $H^1_{\mathrm{FP},\Gamma_-}(\Omega)$ admits a continuous trace operator $\gamma_+ : H^1_{\mathrm{FP},\Gamma_-}(\Omega) \to L^2(\Gamma_+, |{k} \cdot n|)$.
\end{remark}

To later show the uniqueness of the weak solution in \cref{sect:var_form}, we also need to verify the existence of a global trace and the integration by parts formula for certain functions in $H^1_{\mathrm{FP}}(\Omega)$ with vanishing trace on $\Gamma_-$, but not necessarily in $H^1_{\mathrm{FP},\Gamma_-}(\Omega)$.
This is established for spaces where the advective or kinetic terms lie in $L^2(\Omega)$ (see{,} e.g.{,} \cite[Thm.\ 2.2, Prop.\ 2.5]{Bar70}), \cite[Chap. XXI, Remark 3]{DauLio1993_Vol6}). 
Similar or even stronger results for respective functions in  $H^1_{\mathrm{FP}}(\Omega)$ are claimed to be proven in \cite{AM2019,BalPal2020,Car1998}, however, we believe the arguments to be incomplete, for more details see {the supplementary material}.

Since we were not able to prove the existence of a global trace for $H^1_{\mathrm{FP}}(\Omega)$ functions with vanishing trace on the inflow or the outflow boundary, we will formulate the exact result needed for uniqueness of the weak solution as an assumption in \cref{sect:var_form}.

\section{Variational formulation} \label{sect:var_form}
In this section, we develop a variational formulation for \cref{strong_form_with_q} and show its well-posedness. 

Let ${a} : \Omega_{t,x} \times V \times V \to \R$ be a potentially $(x,t)$-dependent bilinear form defined on the velocity space $V$. Moreover, let ${a}$ satisfy the following assumptions:
\begin{align}
&\text{the map } (t,x) \mapsto {a}((t,x); \phi,\psi) \text{ is measurable on } \Omega_{t,x} \text{ for all } \phi,\psi \in V, \label{ass_a_meas}\\
&{a}((t,x); \cdot, \cdot) \text{ is bilinear for a.e. } (t,x) \in \Omega_{t,x}, \label{ass_a_bil}\\
&{a}((t,x); \phi, \psi) \leq {c_a} \|\phi\|_V \|\psi\|_V \text{ with } {c_a} < \infty \text{ for all } \phi,\psi \in V, \text{ a.e. } (x,t) \in \Omega_{t,x}, \label{ass_a_cont}\\
&{a}((t,x); \phi, \phi) + {\lambda_a} \|\phi\|_{L^2(\Omega_v)}^2 \geq {\alpha_a} \|\phi\|_V^2 \text{ with } {\lambda_a} \in \R, {\alpha_a} > 0 \label{ass_a_gard}\\
&\hspace{6.2cm}\text{ for all } \phi \in V, \text{ a.e. } (x,t) \in \Omega_{t,x} . \nonumber 
\end{align}
Note that ${c_a}, {\lambda_a},$ and ${\alpha_a}$ are assumed to be independent of $(x,t)$.

\begin{example}
	For the strong form of the Fokker-Planck equation \cref{strong_form_with_q}, ${a}$ is given for all $\phi, \psi \in V, \text{ a.e.\ } x \in \Omega_x$ by
	\begin{align*}
	{a}(x;\phi,\psi) &= \left (\nabla_{v}\left (q(x,v)^{-1} \phi(v)\right ) ,\nabla_{v} \psi(v) \right )_{L^2(\Omega_v)} \\
	&= \left (q(x,v)^{-1} \nabla_{v} \phi(v) ,\nabla_{v} \psi(v) \right )_{L^2(\Omega_v)} 
	+ \left (\nabla_{v} q(x,v)^{-1}  \phi(v) ,\nabla_{v} \psi(v) \right )_{L^2(\Omega_v)},
	\end{align*}	
	where $\nabla_{v}$ is the tangential gradient on $\Omega_v$, see{,} e.g.{,} \cite{DE2013} for a formal definition. 
	If $q^{-1} \in L^\infty(\Omega_x \times \Omega_v)$ with $\nabla_{v} q^{-1} \in L^\infty(\Omega_x \times \Omega_v)$ and $q^{-1}(x,v) \geq l_q > 0$ for a.e.\ $(x,v)$, then ${a}$  fulfills the conditions \cref{ass_a_meas,ass_a_bil,ass_a_cont,ass_a_gard}, for instance, with ${c_a} = \|q^{-1}\|_{L^\infty} + \|\nabla_{v}q^{-1}\|_{L^\infty}$,
	${\alpha_a} = \tfrac{1}{2} l_q$, and 
	${\lambda_a} = \|\nabla_{v}q^{-1}\|_{L^\infty}^2 / (2l_q) + \tfrac{1}{2} l_q$.
	% Computation of alpha and lambda by Youngs inequality with eps such that term on ||grad phi||^2 is l_q - 1/2 l_q
	Depending on $q$, other estimates might be better, e.g.{,} for $q=q(x)$ and thus $\nabla_{v}q = 0$ we can get ${\alpha_a} = {\lambda_a} = l_q$.
\end{example}

Recalling the function spaces introduced in \cref{eq:L2OmegaxtV,def:HkinG+-}, we define the {space-time-velocity} trial and test spaces as
\begin{equation} \label{def_X_Y}
\cX := L^2(\Omega_{t,x},V), \qquad 
\cY := H^1_{\mathrm{FP},\Gamma_+}(\Omega).
\end{equation}
{ with squared norms (cf. \eqref{eq:L2OmegaxtV}, \eqref{eq:h1kin_norm})
	\begin{align}
	\|w\|_{\cX}^2 & = \int_{\Omega_{t,x}} \|w(t,x)\|_V^2 \dd (t,x) \label{eq:norm_cX},\\
	\|p\|_{\cY}^2 &= 
	\|p\|_{\cX}^2 + \|{k} \cdot \nabla_{t,x}  p \|_{\cX'}^2. \label{eq:norm_cY}
	\end{align}
}
We then define the full bilinear form $b: \cX \times \cY \to \R$ for  $w \in \cX, p \in \cY$ by
\begin{equation} \label{eq:b_bil_form}
b(w,p) := \int_{\Omega_{t,x}} \la w(t,x), -{k(t,x)} \cdot \nabla_{t,x} p(t,x) \ra_{V, V'}
+  a((t,x); w(t,x),p(t,x)) \dd (t,x).
\end{equation}
The functional $f: \cY \to \R$ containing the boundary condition \mbox{$g \in L^2(\Gamma_-, |{k} \cdot n|)$} is given as
\begin{equation*}
f(p) :=  \int_{\Gamma_-} g p
\left | {k}
\cdot n \right | \dd ((t,x),v) \quad \forall \, p \in \cY,
\end{equation*}
which is well-defined due to \cref{prop:trace_PI}, and we thus have  $f \in \cY'$. 

We call $u \in \cX$ a weak solution of \cref{strong_form_with_q}, if
\begin{equation} \label{eq:var_form}
b(u,p) = f(p) \quad \forall \, p \in \cY.
\end{equation}
In the following, we examine the well-posedness of the variational formulation, using the Banach-Ne\v{c}as-Babu\v{s}ka (or inf-sup) Theorem (see e.g.\ \cite[Thm.\ 2.6]{EG2004}). We first prove existence of a weak solution in \cref{subsec:existence}. Then, in \cref{subsec:uniqueness} we also show uniqueness of the weak solution under an additional assumption on the trace of certain $H^1_{\mathrm{FP}}(\Omega)$-functions.

\subsection{Existence of a weak solution} \label{subsec:existence}

We show the existence of a weak solution $u$ to \eqref{eq:var_form} by verifying a dual inf-sup condition. 
To that end, we construct stable pairs of trial and test space functions such that the application of the bilinear form to the function pairs can be estimated from below by the respective norms of the functions.
In these pairs, the trial space functions are derived from the test space functions by the application of the kinetic transport operator and the inverse elliptic velocity operator.
We thus generalize similar proofs for parabolic equations \cite{EG2004,SS2009}{, where} a time derivative {was used} instead of the kinetic transport operator{,} and for transport equations{, where} only an application of the transport operator {was used} \cite{BSU2019,DHSW2012,DemGop11}.

\begin{theorem}\label{prop:inf_sup}
	The bilinear form $b$ satisfies the dual inf-sup condition
	\begin{equation*}
	\inf_{\substack{p \in \cY\\{p \neq 0}}} \sup_{\substack{w \in \cX\\{w \neq 0}}} \frac{b(w,p)}{\|w\|_\cX \|p\|_\cY} \geq \beta
	\end{equation*}
	with an inf-sup constant
	\begin{align}
	\beta &\geq \frac{{\alpha_a}}{\sqrt{2}\max\{1,{c_a}\}}, 
	&& \text{if } {\lambda_a} \leq 0, \label{eq:inf-sup-coercive}\\
	\beta &\geq \frac{{\alpha_a}}{\sqrt{2}\max\{1, {c_a}+{\lambda_a}\}} \frac{e^{-{\lambda_a} T}}{\sqrt{\max\{ 1+2{\lambda_a^2}, 2 \}}}, \hspace{-1.5cm}
	&& \text{if } {\lambda_a} > 0. \label{eq:inf-sup-lambda}
	\end{align}
	Consequently, the variational formulation \eqref{eq:var_form} has at least one weak solution $u \in \cX$.	
\end{theorem}
\begin{remark}
	The estimates for $\beta$ are not worse than estimates for space-time variational formulations for parabolic equations {from \cite{SS2009}}. In fact, for {${\lambda_a} \leq 0$} and assuming ${\alpha_a} \leq 1$ and ${c_a} \geq 1${,} the estimate in \cite[(A.6)]{SS2009} roughly translates\footnote{More precisely, using the notation of this paper, the complete estimate in \cite[(A.6)]{SS2009} reads $\beta_{parab} \geq \min(\alpha_a/c_a^2, \alpha_a)\, / \, (2 \max(\alpha_a^{-2},1) + M_e^2)^{1/2}$, where $M_e$ is an additional positive constant that appears due to a different boundary treatment and that we can leave out here.} to $\beta_{parab} \geq {\alpha_a^2} / (\sqrt{2} {c_a^2})$, while we have $\beta \geq {\alpha_a} / {\sqrt{2}{c_a}}$. The exponential dependence on the final time $T$ for the non-coercive case is the same for both types of equations. 
\end{remark} 
\begin{proof}[Proof of \cref{prop:inf_sup}]
	We start with the case of ${a}$ being coercive, i.e., ${\lambda_a} \leq 0$; the non-coercive case will be treated afterwards via a temporal transformation.

	To show the inf-sup condition we combine ideas from well-posedness results for parabolic equations as e.g.\ in \cite{EG2004, SS2009} and for transport equations as{,} e.g.{,} in \cite{BSU2019}. 
	To that end, we take $0 \neq p \in \cY$ arbitrary, but fixed. We want to construct a suitable $w_p \in \cX$ and show $b(w_p, p) \geq  \beta \|w_p\|_\cX \|p\|_\cY$ for a constant $\beta$ independent of $p$, which makes $\beta$ a lower bound for the inf-sup constant.
	
	Since $p \in \cY$, we have ${r_p} := -{k} \cdot \nabla_{t,x} p \in L^2(\Omega_{t,x};V') = \cX'$.
	Similar to \cite[pp.\ 235]{LioMag1972}, we define the bilinear form $m: \cX \times \cX \to \R$ by 
	\begin{equation*} %\label{def:m_bilf}
	m(w_1,w_2) := \int_{\Omega_{t,x}} {a}((t,x); w_1(t,x), w_2(t,x)) \dd (t,x), \quad \forall  w_1,w_2 \in \cX.
	\end{equation*}
	Since the function $(t,x) \mapsto {a}((t,x); \phi, \psi)$ is assumed to be measurable for all $\phi,\psi \in V$ (see \eqref{ass_a_meas}) and $a((t,x),\cdot,\cdot)$ is continuous and coercive with constants ${c_a}$, ${\alpha_a}$ independent of $(t,x)$ (\cref{ass_a_cont} and \cref{ass_a_gard} with ${\lambda_a} \leq 0$), $m$ is well-defined{,} continuous{,} and coercive over $\cX\times\cX$ with constants ${c_a}$ and ${\alpha_a}$.  
	Therefore, by the Lax-Milgram theorem {it} exists a unique $z_p \in \cX$ with 
	\begin{equation} \label{eq:def_z_p}
	m(z_p, w) = \la {r_p}, w \ra_{\cX', \cX} \quad \forall w \in \cX.
	\end{equation}
	Due to the definitions of $z_p$, ${r_p}$, and $m$, there holds\footnote{In the following, we omit the $(t,x)$ dependence in the integrals.}
	\begin{equation}\label{eq:a_z_p_relation}
	\int_{\Omega_{t,x}} {a}(z_p, w) \dd (t,x) = \int_{\Omega_{t,x}} \la -{k} \cdot \nabla_{t,x} p, w \ra_{V'\!, V} \dd (t,x) \quad \forall w \in \cX.
	\end{equation}
	We now define $w_p := p + z_p \in \cX$.
	To bound $ b(w_p, p)$ from below we use \cref{eq:a_z_p_relation} for {$w=w_p$},
	and the integration by parts formula from \cref{prop:trace_PI}:
	\begin{equation}\label{eq:bwpp}
	\begin{aligned} 
	\hspace{-0.9cm}b(w_p&, p) 
	= \int_{\Omega_{t,x}}\!\! \la p+ z_p, -{k} \cdot \nabla_{t,x} p \ra_{V, V'} +  {a}(p + z_p, p) \dd (t,x) \\
	&= \int_{\Omega_{t,x}}\!\! \la p, -{k} \cdot \nabla_{t,x} p \ra_{V , V'} +  {a}(z_p, z_p) + {a}(p, p)
	+  \la -{k} \cdot \nabla_{t,x} p, p \ra_{V'\!, V} \dd (t,x)\hspace{-0.9em} \\
	&\geq {\alpha_a}(\|p\|_{\cX}^2 + \|z_p\|_{\cX}^2)
	+ 2  \int_{\Omega_{t,x}}\!\! \la -{k} \cdot \nabla_{t,x} p, p\ra_{V'\!, V} \dd (t,x).  \\
	& = {\alpha_a}(\|p\|_{\cX}^2 + \|z_p\|_{\cX}^2)
	+ \int_{\Gamma_-}\!\! p^2  \left | {k} \cdot n \right | \dd s \quad \geq \quad {\alpha_a}(\|p\|_{\cX}^2 + \|z_p\|_{\cX}^2). 
	\end{aligned}
	\end{equation}
	Since we have
	$\la {r_p}, w \ra_{\cX',\cX} = m(z_p, w) \leq  {c_a} \|z_p\|_\cX \|w\|_\cX$ for all $w \in \cX$, {it} holds 
	\begin{equation}\label{eq:est_f_p_2}
	\|{r_p}\|_{\cX'} \leq  {c_a} \|z_p\|_\cX.
	\end{equation}
	Using the definition of $w_p$, ${r_p}$, and the norm of $\cY$ as defined in \cref{eq:h1kin_norm}, we can then estimate
	\begin{equation}\label{eq:new_b_final_est}
	\begin{aligned}
	\|w_p\|_\cX \|p\|_\cY
	&=
	\|p + z_p\|_\cX
	\left (\|p\|_\cX^2 + \|{r_p}\|_{\cX'}^2  \right )^{1/2} \\
	%&= 
	%\left ( \|p + z_p\|_\cX^2 
	%\left (\|p\|_\cX^2 + \|f_p\|_{\cX'}^2  \right ) \right )^{1/2} \\
	&\!\!\!\overset{\eqref{eq:est_f_p_2}}{\leq} 
	\left [ \|p + z_p\|_\cX^2 
	\left ( \|p \|_\cX^2 +  {c_a^2}\|z_p \|_\cX^2 \right ) \right ]^{1/2}\\
	&\leq \left [ 2 \left ( \|p\|_\cX^2 + \|z_p\|_\cX^2 \right ) 
	\left ( \|p \|_\cX^2 +  {c_a^2}\|z_p \|_\cX^2 \right ) \right ]	^{1/2}\\
	%&\leq \left (2 \max\{1, c_v^2\} \left ( \|p \|_\cX^2 + \|z_p \|_\cX^2 \right )^2 \right )^{1/2}\\
	&\leq \sqrt{2}\max\{1,  {c_a}\} \left ( \|p \|_\cX^2 + \|z_p \|_\cX^2 \right ) 
	\overset{\eqref{eq:bwpp}}{\leq} \frac{\sqrt{2}\max\{1, {c_a}\}}{{\alpha_a}}  b(w_p,p).
	\end{aligned}
	\end{equation}
	Since $p \in \cY$ was chosen arbitrarily, we thus have
	\begin{equation}
	\inf_{p \in \cY} \sup_{w \in \cX} \frac{b(w,p)}{\|w\|_{\cX}\|p\|_{\cY}} \geq 
	\beta := \frac{{\alpha_a}}{\sqrt{2}\max\{1, {c_a}\}},
	\end{equation}
	i.e., the claim for coercive ${a}$. 
	
	To address the case that ${a}$ fulfills the G\aa rding inequality \cref{ass_a_gard} with ${\lambda_a} > 0$, we use a standard temporal transformation of the full problem as proposed e.g.\ in \cite{SS2009, UP2014}. We set $\hat w := e^{-{\lambda_a} t}w$ for $w \in \cX$, $\hat p = e^{{\lambda_a} t}p$ for $p \in \cY$, and define the bilinear form $\hat b :\cX \times \cY \to \R$ by
	\begin{equation} \label{eq:temp_transf}
	\hat b(\hat w, \hat p) 
	:= \int_{\Omega_{t,x}}\!\! \la \hat w, - {k} \cdot \nabla_{t,x} \hat p \ra_{V, V'}
	+ {a}((t,x);\hat w, \hat p) + {\lambda_a} (\hat w, \hat p)_{L^2(\Omega_v)} \dd (t,x).
	\end{equation}
	Then it holds $b(w,p)=\hat b(\hat w,\hat p)$ for all $w \in \cX, p \in \cY$. The transformed bilinear form $\hat b$ {is the same as} $b${,} with {a} transformed velocity bilinear form ${\hat a}: V \times V \to \R$ defined by ${\hat a}((t,x);\phi,\psi) = {a}((t,x);\phi,\psi) + {\lambda_a} (\phi, \psi)_{L^2(\Omega_v)}$ for $\phi,\psi \in V$. Due to the G\aa rding inequality \cref{ass_a_gard} and continuity \cref{ass_a_cont} of ${a}$, ${\hat a}$ is coercive with constant ${\hat\alpha_a} = {\alpha_a}$ and continuous with constant ${\hat c_a} = {c_a} + {\lambda_a}$.  
	As in \cite{SS2009}, we can estimate the norms of $\hat w \in \cX$ and $\hat p \in \cY$ by
	\begin{align*}
	\|\hat w\|_\cX  \geq e^{-{\lambda_a} T} \|w\|_{\cX}, \qquad
	\|\hat p \|_\cY \geq \left (\max\{ 1+2{\lambda_a^2}, 2 \} \right )^{-\frac{1}{2}} \|p\|_{\cY}, 
	\end{align*}
	where we use $\|\psi\|_{V'} \leq \|\psi\|_{L^2(\Omega_v)} \leq \|\psi\|_{V}$ for the estimation of the $\cY$-norm. 
	
	Then, the dual inf-sup constant of $b$ can be bounded from below as follows
	\begin{align*}
	\inf_{p \in \cY} \sup_{w \in \cX} \frac{b(w,p)}{\|w\|_\cX \|p\|_\cY} 
	&= \inf_{\hat p \in \cY} \sup_{\hat w \in \cX} \frac{\hat b(\hat w, \hat p)}{\|\hat w\|_\cX \|\hat p\|_\cY} \frac{\|\hat w\|_\cX}{\|w\|_\cX} \frac{\|\hat p\|_\cY}{\|p\|_\cY} \\
	&\geq \frac{{\alpha_a}}{\sqrt{2}\max\{1, {c_a}+{\lambda_a}\}} \frac{e^{-{\lambda_a} T}}{\sqrt{\max\{ 1+2{\lambda_a^2}, 2 \}}}.
	\end{align*}
	Since the dual inf-sup condition implies surjectivity of the operator $B : \cX \to \cY'$ defined by $\la B \cdot , \cdot \ra_{\cY',\cY} = b(\cdot,\cdot)$ and thus existence of a weak solution to \eqref{eq:var_form} (see for instance \cite[Lemma A.40, Remark A.41]{EG2004}), this concludes the proof.
\end{proof}

\subsection{Uniqueness of the weak solution} \label{subsec:uniqueness}

As already mentioned in \cref{sect:function_spaces}, we were not able to prove all necessary trace results in our specific function space. To show uniqueness of the weak solution, we therefore assume the following:
\begin{assumption} \label{ass:global_trace}
	Let $w \in H^1_{\mathrm{FP}}(\Omega)$ {such that} $w=0$ a.e.\ on $\Gamma_-$ {and} $b(w,p) = 0$ for all $p \in \cY$. 
	Then, we {assume this implies} $w \in L^2(\partial\Omega, |{k}\cdot n|)$ and the integration by parts formula
	\begin{equation} \label{eq:integration_by_parts}
	\int_{\Omega_{t,x}} \la {k} \cdot \nabla_{t,x} w , w \ra_{V'\!, V} \dd (t,x) = \tfrac{1}{2} \int_{\partial\Omega} w^2  {k} \cdot n \dd s 
	\end{equation}
	holds. 
\end{assumption}
As discussed in more detail in {the supplementary material, we do not know how to prove \cref{ass:global_trace}, since, for instance, ideas from existing approaches for the related space $H^1_\mathrm{NT}(\Omega) = \{ w \in L^2(\Omega) : {k} \cdot \nabla_{t,x} w \in L^2(\Omega)\}$ cannot {be} readily transferred to the $H^1_{\mathrm{FP}}(\Omega)$ case. We therefore leave it as an open problem. We emphasize that the respective trace and integration by parts result holds for all $H^1_\mathrm{NT}(\Omega)$-functions with zero inflow or outflow trace (cf.\ \cite{Bar70,Cessenat1984,Cessenat1985},\cite[Chap.\ XXI]{DauLio1993_Vol6}), and also for all $H^1_{\mathrm{FP}}(\Omega)$-functions that can be approximated by smooth functions vanishing on the inflow or outflow boundary (\cref{prop:trace_PI}).
Additionally, \cref{ass:global_trace} {only refers} to $H^1_{\mathrm{FP}}(\Omega)$-functions with vanishing trace on $\Gamma_-$ and satisfying a weak form of the differential equation with zero boundary condition. This additional condition on the considered functions might make it possible to show and exploit a higher regularity of the considered functions to prove existence of suitable traces and \eqref{eq:integration_by_parts}.

We now show uniqueness of the weak solution in the form of surjectivity of the dual operator. To that end, we follow the general structure of respective proofs for parabolic equations \cite[Thm 6.6, p.\ 283]{EG2004} and transport equations \cite[Thm.\ 16]{Azerad1996}{.} We take a function $w \in \cX$ solving \eqref{eq:var_form} with zero right-hand side and prove that $w=0$ by showing that $w$ possesses space- and time derivatives, that $w$ has trace zero on the outflow boundary, and finally that $w$ must therefore vanish on the whole domain.
\begin{theorem} \label{prop:surjectivity}
	If \cref{ass:global_trace} holds, then for all $0 \neq w \in \cX$ we have
	\begin{equation*}
	\sup_{p \in \cY} b(w,p) > 0.
	\end{equation*}
\end{theorem}
\begin{proof}	
	Let $w \in \cX$ such that 
	\begin{equation} \label{eq:inj_w}
	b(w,p) = 0 \quad  \forall \, p \in \cY.
	\end{equation}
	To prove the claim, we need to show that $w = 0$. 
	First, we show that $w$ has a weak derivative $-{k} \cdot \nabla_{t,x} w \in \cX' = L^2(\Omega_{t,x}; V')$. To that end, let $\psi \in C_0^\infty(\Omega_{t,x})$ and $\phi \in V$ be arbitrary. Then $\psi \phi = 0$ on ${\hat\Gamma}$, and by approximating $\phi$ in $C^\infty(\Omega_v) $ we see that $\psi \phi \in \cY$.  
	Using the definition of the weak $(t,x)$-derivative and testing \eqref{eq:inj_w} with $p = \psi \phi$ we obtain
	\begin{align*}
	\int_{\Omega_{t,x}} &\left \la {k(t,x)} \cdot \nabla_{t,x} w(t,x), \phi \right \ra_{V'\!,V} \psi(t,x) d(t,x)  \\
	&= - \int_{\Omega_{t,x}} \left \la w(t,x), {k(t,x)} \cdot \nabla_{t,x} \psi(t,x) \phi \right \ra_{V,V'} \dd (t,x) \\
	&= - \int_{\Omega_{t,x}} {a}((t,x); w(t,x), \psi(t,x) \phi) \dd (t,x)\\
	&= - \int_{\Omega_{t,x}} \la{ A}(t,x) w(t,x), \phi \ra_{V'\!,V} \psi(t,x) \dd (t,x),
	\end{align*}
	where the operator $A_v(t,x) \in \mathcal{L}(V,V')$ is defined as $\la A_v(t,x) \phi, \rho \ra_{V'\!,V} = {a}((t,x);\phi,\rho)$ for all $\phi,\rho \in V$, a.e.\ $(t,x) \in \Omega_{t,x}$. 
	Due to the density of $C_0^\infty(\Omega_{t,x})$ in $L^2(\Omega_{t,x})$ 
	have 
	\begin{equation} \label{eq:w_prop_V'}
	- {k} \cdot \nabla_{t,x} w = A_v w \in \cX',
	\end{equation}
	which especially means that $w \in H^1_{\mathrm{FP}}(\Omega)$. 
	
	Next, let $K \subset\subset \Gamma_-$ be an arbitrary but fixed compactly embedded subset of $\Gamma_-$. Moreover, let $z \in C^\infty(\bar\Omega)$ with $z = 0$ on ${\hat\Gamma} \setminus K$.
	We show $wz \in \cY$: Since $w \in H^1_{\mathrm{FP}}(\Omega)$, due to \cref{prop:density} there is a sequence $(w_n)_{n \in \N} \subset C^\infty(\bar\Omega)$ with $\|w_n - w\|_{H^1_{\mathrm{FP}}(\Omega)} \overset{n \to \infty}{\to} 0$. Therefore, we have $w_n z \in C^\infty(\bar\Omega)$ with $wz = 0$ on $\Gamma_+$. 
	Due to \cref{lem:product_continuous}, it holds
	\begin{equation*}
	\|wz - w_n z\|_{H^1_{\mathrm{FP}}(\Omega)}
	\leq C \|z\|_{C^1(\Omega)} \|w-w_n\|_{H^1_{\mathrm{FP}}(\Omega)} 
	\end{equation*}
	and thus $w_n z \to wz$ in $H^1_{\mathrm{FP}}(\Omega)$ as $n \to \infty$. Invoking the definition of $\cY$ in \eqref{def_X_Y},\eqref{def:HkinG+-} we obtain  $wz \in \cY$. 
	
	Since $K \subset \Gamma_-$ is compact, we may apply \cref{lem:trace} to infer that $w$ has a trace on $K$ and $w|_{K} \in L^2(K, |{k} \cdot n|)$. Thanks to $z|_{{\hat\Gamma}} \in L^\infty({\hat\Gamma})$ and $\supp z|_{{\hat\Gamma}} \subset K$, we have
	\begin{equation*}
	\left | \int_{{\hat\Gamma}} w^2 z \left |{k} \cdot n \right |\dd s \right |
	= \left | \int_{K} w^2 z \left |{k} \cdot n \right |\dd s \right |
	\leq \|z\|_{L^\infty(K)} \|w\|_{L^2(K, |{k} \cdot n|)}^2 < \infty.
	\end{equation*}
	
	As a consequence we can apply the linear functional in \eqref{eq:w_prop_V'} to $wz \in \cY \subset \cX$, perform integration by parts, since the boundary integral exists, and use \cref{eq:inj_w}:
	\begin{align*}
	0 &= 
	\int_{\Omega_{t,x}} \la {k} \cdot \nabla_{t,x} w +  A_v w , wz \ra_{V'\!,V} d(t,x) \\
	&= \int_{\Omega_{t,x}} \la w, -{k} \cdot \nabla_{t,x} (wz) \ra_{V,V'} +  {a}(w,wz) d(t,x) 
	+ \int_{{\hat\Gamma}} w^2 z  {k} \cdot n  \dd s \\
	&= \underbrace{b(w,wz)}_{=0} - \int_{K}   w^2 z \left | {k}  \cdot n \right | \dd s
	= - \int_{K} w^2 z \left |{k} \cdot n \right |\dd s.
	\end{align*}
	Since $z|_K \in C^\infty_0(K)$ can be chosen arbitrarily and $|{k} \cdot n| > 0$ on $K$, the fundamental lemma of calculus of variations yields $w = 0$ a.e.\ on $K$. As also $K \subset \Gamma_-$ was chosen arbitrarily, we have $w = 0$ a.e.\ on $\Gamma_-$.
	
	Thanks to \cref{ass:global_trace}, it therefore holds $w \in L^2(\partial\Omega, |{k} \cdot n|)$. We can thus use integration by parts for \cref{eq:w_prop_V'} applied to $w$. Assuming first that ${a}$ is coercive, i.e., ${\lambda_a} \leq 0$, we obtain 
	\begin{align*}
	0 &= \int_{\Omega_{t,x}} \la {k} \cdot \nabla_{t,x} w + A_v w , w \ra_{V'\!,V} \dd(t,x) \\
	&= \int_{\Omega_{t,x}} \la {k} \cdot \nabla_{t,x} w, w \ra_{V'\!,V} \dd(t,x)
	+ \int_{\Omega_{t,x}} {a}(w,w) \dd(t,x) \\
	&\geq \tfrac{1}{2}\int_{\Gamma_+} w^2 \underbrace{{k} \cdot n}_{> 0} \dd s + {\alpha_a} \|w\|_{\cX}^2,
	\end{align*} 
	which implies $w = 0$. 
	
	If ${a}$ is not coercive, we use the temporal transformation described in the proof of \cref{prop:inf_sup}.
	Setting $\hat w = e^{-{\lambda_a} t} w$ and using the definition of $\hat b$ in \eqref{eq:temp_transf}, we see that \eqref{eq:inj_w} is equivalent to $\hat b(\hat w,\hat p) = 0$ for all $\hat p \in \cY$. Since $\hat a$ is coercive, we have proven that $\hat w = 0$ and thus also $w=0$. 	
\end{proof}

We summarize our findings in the following theorem.
\begin{theorem}[Well-posedness]
	There exists a solution $u \in \cX$ to the variational problem \cref{eq:var_form}. If \cref{ass:global_trace} holds, the solution is unique and satisfies the stability estimate
	\begin{equation*}
	\|u\|_\cX \leq \frac{1}{\beta} \|f\|_{\cY'}
	\end{equation*}
	for $\beta$ as defined in \cref{prop:inf_sup}.
\end{theorem}
\begin{proof}
	{
		Standard inf-sup theory ensures the existence of a solution due to the continuity of $b$ and the dual inf-sup condition stated in \cref{prop:inf_sup}. Under \cref{ass:global_trace}, \cref{prop:surjectivity} yields the dual surjectivity, which implies uniqueness and the stability estimate. }
\end{proof}

\section{Discretization} \label{sect:discretization}

We now design a stable and efficient discretization scheme for \eqref{eq:var_form}. To that end, we use a Petrov-Galerkin projection onto
problem-dependent discrete spaces realizing the stable function pairs with test functions $p \in \cY$ and trial functions $w_p \in \cX$ developed in the proof of \cref{prop:inf_sup}. As a result, the discrete inf-sup stability and thus the well-posedness of the discrete problem follow analogously to the continuous results with the same stability constant. We then illustrate for a class of data functions how the trial space functions $w_p^\delta$ can be efficiently computed by solving low-dimensional elliptic problems in the velocity domain.

\subsection{Stable Petrov-Galerkin schemes}

To define an approximation of the solution $u \in \cX$ of \cref{eq:var_form}, we use a Petrov-Galerkin projection onto suitable discrete spaces: Given discrete trial and test spaces $\cXd \subset \cX$ and $\cYd \subset \cY$, the Petrov-Galerkin approximation $u^\delta \in \cXd$ is defined by
\begin{equation} \label{eq:Petrov-Galerkin}
b(u^\delta, v^\delta) = f(v^\delta) \quad \forall v^\delta \in \cYd.
\end{equation} 
Well-posedness then depends on the inf-sup stability of the discrete problem. To find a pair of spaces leading to a stable scheme, we transfer ideas from \cite{BSU2019} to our setting. In \cite{BSU2019}, a stable discretization with a discrete inf-sup constant {equal to} one was built for a transport equation by fixing a discrete test space and defining a problem dependent trial space with optimal stability properties. 
In this manuscript{,} we will use the same strategy: We start with a discrete test space and define the corresponding trial space based on the trial space functions used in the proof of \cref{prop:inf_sup}.

To that end, we first define a discrete space $V_h \subset V$ for the discretization in the  velocity direction. Since the $\cY$-norm contains a term in the $\cX' = L^2(\Omega_{t,x},V')${-}norm {(see \eqref{eq:norm_cY})} which is not computable, we consider the norm
\begin{equation}
\|w\|_{L^2(\Omega_{t,x},\Vhpr)}^2 := \int_{\Omega_{t,x}} \|w(t,x)\|_{\Vhpr}^2 \dd (t,x), \quad \|\psi\|_{\Vhpr} := \sup_{\phi^h \in V_h} \frac{\la \psi,\phi^h \ra_{V'\!,V}}{\|\phi^h\|_V}
\end{equation}
instead of $\|\cdot\|_{L^2(\Omega_{t,x},V')}$ where necessary. 

Let $\cYd \subset \cY$ be a discrete space for which we assume $w^\delta(t,x) \in V_h$ for all $w^\delta \in \cYd$ and a.e.\ $(t,x) \in \Omega_{t,x}$. 
{The space} $\cYd$ will be used as {the} test space for the Petrov-Galerkin approximation. 
We define the discrete version of the $\cY$-norm by
\begin{equation}
\|w\|_{\cYd}^2 := \|w\|_{L^2(\Omega_{t,x},V)}^2 + \|{k} \cdot \nabla_{t,x} w\|_{L^2(\Omega_{t,x},\Vhpr)}^2.
\end{equation}

Since we will make use of the function pairs developed in the proof of \cref{prop:inf_sup}, we assume for the discretization that the velocity bilinear form ${a}$ is coercive, i.e., ${\lambda_a} \leq 0$. For problems, where ${a}$ only satisfies the G\aa rding inequality \eqref{ass_a_gard} with ${\lambda_a} > 0$, a temporal transform{ation} of the problem as described in \cref{sect:var_form} can be performed{. Then,} the transformed problem with a coercive bilinear form ${\hat a}$ can be discretized.

We now define a problem-dependent discrete trial space.
For each $p^\delta \in \cYd$, we denote $f_{p}^\delta := {-} {k} \cdot \nabla_{t,x} p^\delta(t,x) \in\cX'$.  
We then define the function $z_p^\delta \in \cX$ as the solution of
\begin{equation} \label{eq:def_zpdelta}
{a}(z_p^\delta(t,x), \phi^h) = \la {r_p^\delta}(t,x), \phi^h \ra_{V'\!, V}, \quad \forall \phi^h \in V_h, \text{ a.e.\ } (t,x) \in \Omega_{t,x}.
\end{equation}
{The function} $z_p^\delta$ is the discrete counterpart of $z_p$ defined in \eqref{eq:def_z_p}, here {it is} defined pointwise in $\Omega_{t,x}$ due to the discrete setting.
Then, the discrete trial space $\cXd \subset \cX$ is defined as
\begin{align} \label{eq:def_cXd}
\cXd &:= \{ p^\delta + z_p^\delta : p^\delta \in \cYd \}.
\end{align}

\begin{proposition} \label{prop:discrete_well_posed}
	If ${\lambda_a} \leq 0$ in \cref{ass_a_gard} and thus {${a}$ is coercive}, and
	if the discrete trial and test spaces $\cXd$ and $\cYd$ are chosen according to \eqref{eq:def_cXd}, then there exists a unique solution $u^\delta \in \cXd$ to \eqref{eq:Petrov-Galerkin}. {Moreover, we have discrete inf-sup estimate
		\begin{equation}\label{eq:discrete_inf_sup} 
		\inf_{\substack{p^\delta \in \cYd\\{p^\delta \neq 0}}} 
		\sup_{\substack{w^\delta \in \cXd\\{w^\delta \neq 0}}}
		\frac{b(w^\delta, p^\delta)}{\|w^\delta\|_{\cX} \|p^\delta\|_{\cYd}}
		\geq \beta_\delta
		\geq {\alpha_a} (\sqrt{2}\max\{1,{c_a}\})^{-1}.
		\end{equation}
	}
\end{proposition}
\begin{remark} \label{remark:temp_transf}
	For {${\lambda_a}>0$} the respective result holds for the discretization of the transformed problem according to \eqref{eq:temp_transf} with ${\hat a}$ being coercive.	
\end{remark}
\begin{proof}
	We can reuse all essential parts of the proof of the inf-sup constant for the continuous problem to also prove discrete inf-sup stability of \eqref{eq:Petrov-Galerkin}.
	
	Let $0 \neq w^\delta \in \cXd$ be fixed. Then, by definition of $\cXd$ there is $p^\delta \in \cYd$ such that $w^\delta = p^\delta + z_p^\delta$ with $z_p^\delta$ defined as in \eqref{eq:def_zpdelta}. 
	By using \eqref{eq:def_zpdelta} and the same arguments as in \eqref{eq:bwpp} we obtain
	\begin{equation} \label{eq:bwpp_discr}
	b(w^\delta, p^\delta) = b(p^\delta + z_p^\delta, p^\delta) \geq {\alpha_a} \left (\|p^\delta\|_\cX^2 + \|z_p^\delta\|_\cX^2 \right ).
	\end{equation}
	As we have 
	\begin{equation*}
	\la {r_p^\delta}(t,x), \phi^h \ra_{V'\!, V} = {a}(z_p^\delta(t,x), \phi^h) \leq {c_a} \|z_p^\delta(t,x)\|_V \|\phi^h\|_V \;\;\forall \phi^h \in V_h, \text{ a.e.}\, (t,x) \in \Omega_{t,x}
	\end{equation*}
	we can inflect that
	\begin{equation} \label{eq:est_f_p_d}
	\|{r_p^\delta}\|_{L^2(\Omega_{t,x},\Vhpr)} \leq {c_a} \|z_p^\delta\|_{\cX}.
	\end{equation}
	Therefore, we obtain analogously to \eqref{eq:new_b_final_est}, but using the discrete $\cYd$-norm,
	\begin{equation}
	\begin{aligned}
	\|w_p^\delta\|_\cX \|p^\delta\|_\cYd
	&=
	\|p^\delta + z_p^\delta\|_\cX
	\left (\|p^\delta\|_\cX^2 + \|{r_p^\delta}\|_{L^2(\Omega_{t,x},\Vhpr)}^2  \right )^{1/2} \\
	%	&= 
	%	\left ( \|p^\delta + z_p^\delta\|_\cX^2 
	%	\left (\|p^\delta\|_\cX^2 + \|f_p^\delta\|_{L^2(\Omega_{t,x},\Vhpr)}^2  \right ) \right )^{1/2} \\
	&\!\!\!\overset{\eqref{eq:est_f_p_d}}{\leq} 
	\left [ \|p^\delta + z_p^\delta\|_\cX^2 
	\left ( \|p^\delta \|_\cX^2 + {c_a^2}\|z_p^\delta \|_\cX^2 \right ) \right ]^{1/2}\\
	&\leq \left [ 2 \left ( \|p^\delta\|_\cX^2 + \|z_p^\delta\|_\cX^2 \right ) 
	\left ( \|p^\delta \|_\cX^2 + {c_a^2}\|z_p^\delta \|_\cX^2 \right ) \right ]^{1/2}\\
	%	&\leq \left (2 \max\{1,c_v^2\} \left ( \|p^\delta \|_\cX^2 + \|z_p^\delta \|_\cX^2 \right )^2 \right )^{1/2}\\
	&= \sqrt{2}\max\{1, {c_a}\} \left ( \|p^\delta \|_\cX^2 + \|z_p^\delta \|_\cX^2 \right ) \overset{\eqref{eq:bwpp_discr}}{\leq} \frac{\sqrt{2}\max\{1,{c_a}\}}{{\alpha_a}} b(w_p^\delta,p^\delta).
	\end{aligned}
	\end{equation}	
	This means that $b$ is inf-sup stable on the spaces $(\cXd,\|\cdot\|_{\cX}), (\cYd, \|\cdot\|_{\cYd})$ with constant $\beta_\delta \geq {\alpha_a} (\sqrt{2}\max\{1,{c_a}\})^{-1}$. Since for all $0\neq p^\delta$ it holds $b(w_p^\delta,p^\delta) > 0$ and thus $w_p^\delta \neq 0$, we have $\dim(\cXd) = \dim(\cYd)$. Therefore, inf-sup stability already guarantees well-posedness of the discrete problem \eqref{eq:Petrov-Galerkin}.
\end{proof}
\begin{remark}
	Due to the finite-dimensional spaces, the Petrov-Galerkin approximation $u^\delta \in \cXd$ is unique even if \cref{ass:global_trace} does not hold.
\end{remark}

{
	\begin{remark}[Choice of $\lambda_a$ in the case $\lambda_a > 0$] \label{remark:lambda_a}
		For possibly non-coercive problems, there is usually some flexibility in the choice of $\alpha_a$ and $\lambda_a$ such that the G\aa rding inequality \eqref{ass_a_gard} is fulfilled: On the one hand, if \eqref{ass_a_gard} holds for a specific $\lambda_a$, all $\tilde \lambda_a > \lambda_a$ are also possible. On the other hand, often \eqref{ass_a_gard} holds for all $\lambda_a > 0$ with different respective $\alpha_a > 0$; think, for instance, of $a(\psi,\theta) = (\nabla_v \psi, \nabla_v \theta)_{L^2(\Omega_v)}$, where \eqref{ass_a_gard} holds for any $\lambda_a > 0$ with $\alpha_a = \min(1,\lambda_a)$. When using a temporal transformation before the discretization, the constant $\lambda_a$ should not be too large: Since $e^{\lambda_a t}$ appears in the temporal transformation, a large $\lambda_a$ leads to error amplification and a very small effective inf-sup constant of the ``non-transformed'' discrete problem (cf. \cref{eq:inf-sup-lambda}). Therefore, a suitable balancing of $\lambda_a$ and $\alpha_a$ with possibly small $\lambda_a$ and large $\alpha_a$ should be sought to obtain a stable discretization when using the temporal transformation.  
	\end{remark}
}

\subsection{Efficient numerical scheme} \label{sec:comp}

Regarding the computational realization of the Petrov-Galerkin approximation, we have to take into account the specific choice of the discrete spaces according to \cref{eq:def_cXd}. To assemble the linear system and to represent the discrete solution, the functions $z_p^\delta$ defined by \cref{eq:def_zpdelta}, have to be computed for all basis functions of $\cYd$. 
We illustrate how this can be done very efficiently for the case where ${a}$ is coercive and has the separable form
\begin{equation}\label{eq:a_v_sep}
{a}((t,x),\phi,\psi) = d(t,x)  {\tilde a}(\phi,\psi),
\end{equation}
where $d \in L^\infty(\Omega_{t,x})$ satisfies $d(t,x)\geq \alpha^d > 0$ for a.e.\ $(t,x) \in \Omega_{t,x}$ and $ {\tilde a} : V \times V \to \R$ is a coercive bilinear form. {   }

To build the discrete test space, let first  $\bar \cYd^{\!t,x} \subset H^1(\Omega_{t,x})$ be a discrete space in the space-time domain with basis $(p^{t,x,\delta}_i(t,x))_{i=1}^{n_{t,x}}$ and let $V_h \subset V$ be the already defined velocity discrete space with basis $(\psi^h_j(v))_{j=1}^{n_v}$. Denoting the tensor product of these spaces by $\bar \cYd := \bar \cYd^{\!t,x} \otimes V_h$, we then set
\begin{equation*}
\cYd := \spanlin \{p_{i,j}^\delta = p^{t,x,\delta}_i \psi^h_j :  p_{i,j}^\delta|_{\Gamma_+} = 0 \} \subset \bar \cYd \cap \cY.
\end{equation*}
We may then use this tensor product structure to efficiently solve \cref{eq:def_zpdelta}: Fixing a basis function $p_{i,j}^\delta = p^{t,x,\delta}_i \psi^h_j$ of $\cYd$, the right-hand side of \cref{eq:def_zpdelta} reads
\begin{align*}
\la - {k} \cdot \nabla_{t,x} p_{i,j}^\delta(t,x), \phi^h  \ra_{V'\!,V} 
= &-\partial_{t} p^{t,x,\delta}_i(t,x) \!
\int_{\Omega_v}\!\!\! \psi^h_j(v) \phi^h(v) \dd v \\
&- \sum_{k = 1}^d
\partial_{x_k} p^{t,x,\delta}_i(t,x) \!
\int_{\Omega_v}\!\!\! v_k \psi^h_j(v) \phi^h(v) \dd v 
\end{align*}
for all $\phi^h \in V_h$, a.e.\ $ (t,x) \in \Omega_{t,x}$.
Using the separable form of ${a}$ \cref{eq:a_v_sep}, we can rewrite \cref{eq:def_zpdelta} as follows: Find $z_{i,j}^\delta := z_{p_{i,j}^\delta}^\delta \in \cX$, such that
\begin{align*}
d(t,x)  {\tilde a}(z_{i,j}^\delta(t,x),\phi^h)
&= - \partial_{t} p^{t,x,\delta}_i(t,x) 
\int_{\Omega_v}\!\!\! \psi^h_j(v) \phi^h(v) \dd v \\
&\quad - \sum_{k = 1}^d \partial_{x_k} p^{t,x,\delta}_i(t,x) 
\int_{\Omega_v}\!\!\! v_k \psi^h_j(v) \phi^h(v) \dd v \\
&\hspace{4cm}\forall \phi^h \in V_h, \text{ a.e.\ } (t,x) \in \Omega_{t,x}.
\end{align*}
Hence, the computation of all $z_{i,j}^\delta$ can be separated in the following way: We first compute the solutions $\rho_j^{1}, \rho_j^{v_1}, \dots, \rho_j^{v_d} \in V_h$ to the problems
\begin{equation}
\begin{split}
{\tilde a}(\rho_j^{1},\phi^h) &= \int_{\Omega_v}\!\!\! \psi_j^h(v) \phi^h(v) \dd v, \quad \forall \phi^h \in V_h, \\
{\tilde a}(\rho_j^{v_k},\phi^h) &= \int_{\Omega_v}\!\!\! v_k \psi_j^h(v) \phi^h(v) \dd v, \quad \forall \phi^h \in V_h, k=1,\dots, d,
\end{split}
\end{equation}
for all basis functions $\psi_j^h \in V_h, j=1,\dots,n_v$. 
Then, the $z_{i,j}^\delta$ are given by
\begin{equation}\label{eq:zpdelta_composed}
z_{i,j}^\delta(t,x,v) = - d(t,x)^{-1} \left ( \partial_{t} p^{t,x,\delta}_i(t,x) \rho_j^{1}(v) + \sum_{k = 1}^d \partial_{x_k} p^{t,x,\delta}_i(t,x) \rho_j^{v_k}(v) \right ).
\end{equation}
The full solution process thus consists of the following steps:
\begin{enumerate}
	\item Precompute $\rho_j^1, \rho_j^{v_k}$, i.e., solve $(d+1)\times n_v$ problems of size $n_v$, which can be done in parallel.
	\item Assemble the stiffness matrix $[b(p_{i,j}^\delta+z_{i,j}^\delta, p_{k,l}^{\delta})]_{(k,l),(i,j)}$, using \cref{eq:zpdelta_composed}, and assemble the load vector $[f(p_{k,l}^{\delta})]_{(k,l)}$.
	\item Solve the linear system of equations to obtain the coefficient vector $[u_{i,j}]_{(i,j)}$.
	\item Compose the solution $u^\delta = \sum_{i,j} u_{i,j} (p_{i,j}^\delta+z_{i,j}^\delta) \in \cXd$ by again using  \cref{eq:zpdelta_composed} { for $z_{i,j}^\delta$.}
\end{enumerate}
Compared to using finite element spaces without any stabilization, the additional costs thus only lie in the $n_v$-sized problems (step 1) and possibly more nonzero elements in the stiffness matrix. 
These effects only depend on the dimension $n_v$ of $V_h$. Therefore, the proposed discretization strategy is especially well-suited for using specific spaces $V_h$ of low dimension, which can be achieved for example by using polynomial bases or a hierarchical model reduction approach as proposed in \cite{BLOS16}. 

In order to efficiently compute the problem-dependent basis functions, we heavily rely on the separable form of the bilinear form ${a}$ given in \eqref{eq:a_v_sep}, { which is unfortunately often not fulfilled for realistic data.
	For general bilinear forms, \eqref{eq:def_zpdelta} remains a variational problem in all dimensions that is not directly decomposable in single low-dimensional problems. 
	However, as the velocity operator is elliptic, for realistic data functions we usually expect the problem to be well-suited for model reduction strategies.
	Therefore, it might be possible to use low-rank approximations as done in a related setting in \cite{BiNoZa14} to find sufficiently accurate approximate solutions to \eqref{eq:def_zpdelta} in a computationally efficient manner.}

More generally, due to the high-dimensionality of the problem, it is especially desirable to combine the {approach proposed in this manuscript} with further approximations as the already mentioned hierarchical model reduction \cite{BLOS16} or tensor-based methods that have already been used in similar Petrov-Galerkin settings \cite{BiNoZa14,HPSU2019} and to discretize kinetic equations like the radiative transfer equation \cite{GreSch2011,WHS2008} or the Vlasov equation \cite{EhrLom2017,EinLub2018,Kormann2015}.

\section{Numerical experiments} \label{sect:num_exp}

We investigate the properties of the method developed in \cref{sect:discretization} by implementing the discretization { for the Fokker-Planck equation \cref{eq:fp_intro} on a two-dimensional spatial domain as well as for a modified stationary equation.} We are especially interested in { the convergence of the discretization error,} analyzing how sharp the lower bound for the inf-sup constant is and examining the efficiency in light of the nonstandard discrete spaces . The source code to reproduce all results is provided in \cite{Bru2020zenodo}.

\subsection{Test Cases}

{
	Let $\Omega_x = (0,1)^2 \subset \R^2$ be the spatial domain and $I_t = (0,0.75)$ be the time interval. We parametrize $\Omega_v = S^1$ by the angle $\varphi \in [0,2\pi)$, leading to $v = \tcolvec{\cos \varphi \\ \sin \varphi}$ and $\Delta_{v} u = \frac{\partial^2}{\partial \varphi^2}u$.
	
	We consider the Fokker-Planck equation \eqref{strong_form_with_q} for a constant $q \in \R_{+}$. Then, the equation reads
	\begin{equation}\label{eq:test_case_timdep}
	\begin{aligned}
	\partial_t u((t,x),\varphi) + \tcolvec{\cos \varphi \\ \sin \varphi} \cdot \nabla_x u((t,x),\varphi)  &= q^{-1} \tfrac{\partial^2}{\partial \varphi^2} u((t,x),\varphi)&& \text{in } \Omega, \\
	u((t,x),\varphi) &= g((t,x),\varphi) && \text{on } \Gamma_{\!-},\hspace{-1em}
	\end{aligned}
	\end{equation}
	where we choose the initial condition
	\begin{align*}
	g((0,x),\varphi) &:= 
	\begin{cases}
	\tfrac{1}{2\pi} (128r(x)^3 - 48r(x)^2 + 1),  &r(x) < \tfrac{1}{4}, \\
	0, & r(x) \geq \tfrac{1}{4},
	\end{cases} 
	\end{align*}
	with $r(x_1,x_2) := \sqrt{(0.5-x_1)^2 + (0.5-x_2)^2}$ 
	and zero spatial inflow boundary conditions $g|_{\Gamma_-^x(\varphi)} \equiv 0$ for all $\varphi \in [0,2\pi)$.
	
	The corresponding velocity bilinear form
	\begin{equation*}
	a(\psi,\rho) := q^{-1} \int_0^{2\pi} \psi'(\varphi) \rho'(\varphi) \dd \varphi \quad \forall \psi,\rho \in V = H^1(\Omega_v)
	\end{equation*}
	fulfills the G\aa rding inequality \eqref{ass_a_gard} for any $\lambda_a > 0$ with $\alpha_a = \min(q^{-1},\lambda_a)$. As mentioned in \cref{remark:lambda_a}, a choice with possibly large $\alpha_a$ and possibly small $\lambda_a$ is desirable to obtain good results when using a temporal transformation according to \eqref{eq:temp_transf}.
	We only consider cases where $0.1 \leq q^{-1} \leq 1$, therefore  we select $\lambda_a= q^{-1}$, $\alpha_a= q^{-1}$. Then, we discretize the transformed problem, where the transformed velocity bilinear form  coincides with the scaled $V$-scalar product, i.e., $\hat a = q^{-1}(\cdot,\cdot)_{V}$. 
	
	For the discretization we choose $V_h \subset V$
	as the continuous linear FE space on $[0,2\pi)$ with periodic boundary condition and uniform mesh with size $h_v=2\pi/n_v$. The space $\bar \cYd^{\!t,x} \subset H^1(\Omega_{t,x})$ is chosen as the continuous $\mathbb{Q}_2$ FE space on a 3D rectangular mesh with uniform 1D mesh sizes $h_t=0.75/n_t$ and $h_{x_1}=h_{x_2} = 1/n_x$. 
	The trial space $\cXd$ is computed as described in \cref{sec:comp} by first solving $3 n_v$ problems of dimension $n_v$. From the definition we see that  $\cXd \subset \bar \cXd^{\!t,x} \otimes V_h$, with $\cXd^{\!t,x} \subset L^2(\Omega_{t,x})$ being the respective discontinuous $\mathbb{Q}_2$ FE space. 
	After computing the transformed solution $\hat u^\delta \in \cXd$, we obtain the discrete solution to \eqref{eq:test_case_timdep} by setting $u^\delta := e^{t} \hat u^\delta$. 
	
	To investigate the convergence rate of the newly proposed scheme, we additionally consider a stationary (and thus lower-dimensional) problem with a manufactured solution $u(x_1,x_2,\varphi) = \sin^2(\pi x_1) \sin^2(\pi x_2) \sin^2(\varphi)$ and corresponding right-hand side $f_0$; therefore slightly deviating from the original problem. More precisely, we consider
	\begin{equation}\label{eq:test_case_time_ind}
	\tcolvec{\cos \varphi \\ \sin \varphi} \cdot \nabla _x u(x,\varphi) + c\, u(x,\varphi) = d\,\tfrac{\partial^2}{\partial \varphi^2} u(x,\varphi) + f_0(x,\varphi) \quad \text{in } \Omega_x \times \Omega_v
	\end{equation}
	with reaction and velocity diffusion constants $c,d \in \R$, $c,d>0$ and zero inflow boundary conditions on $\Gamma_- \subset \partial \Omega_x \times \Omega_v$.
	Note that we require $c>0$ here in order to obtain a coercive bilinear form 
	${a}: V \times V \to \R$
	\begin{equation*}
	{a}(\psi,\rho) = \int_0^{2\pi}\!\!\! d\, \psi'(\varphi) \rho'(\varphi) + c\, \psi(\varphi) \rho(\varphi) \dd \varphi \quad \forall \psi,\rho \in V.
	\end{equation*}
	Then, the bilinear form $a$ is coercive with constant ${\alpha_a} = \min(c,d) > 0$ and continuous with constant $\gamma_v = \max(c,d)$. The variational formulation for the stationary equation \eqref{eq:test_case_time_ind} is based on $\cX_{\text{st}} := L^2(\Omega_{x};H^1(\Omega_v))$, and $\cY_{\text{st}} = \clos_{\|\cdot \|_{\cY_{\text{st}}}} \{w \in C^1(\Omega_x \times \Omega_v): w=0 \text{ on } \Gamma_{\text{st},+} \}$, where 
	\begin{align*}
	\Gamma_{\text{st},+} &= \{(x,v) \in \partial \Omega_x \times \Omega_v :  \tcolvec{\cos \varphi \\ \sin \varphi} \cdot n_x > 0\}, \\
	\|w\|_{\cY_{\text{st}}}^2 &= \|w\|_{\cX_{\text{st}}}^2 + \|\tcolvec{\cos \varphi \\ \sin \varphi} \cdot \nabla_x w \|_{\cX_{\text{st}}'}^2.
	\end{align*} 
	The space-velocity bilinear form is
	\begin{align*}
	b_{\text{st}}(w,p) := \int_{\Omega_{x}}\!\!\!  \la  w(x), - \tcolvec{\cos \\ \sin } \cdot \nabla_x p(x) \ra_{V,V'} 
	+ {a}(w(x),p(x)) \dd x, \quad \forall w \in \cX_{\text{st}}, p \in \cY_{\text{st}}
	\end{align*}
	and the functional describing the source term is defined as
	\begin{equation*}
	f_{\text{st}}(p) := \int_{\Omega_x} \int_0^{2\pi} f_0(x,\varphi) p(x,\varphi) \dd \varphi \dd x \quad \forall p \in \cY_{\text{st}}.
	\end{equation*}
	
	Well-posedness of the weak formulation of \cref{eq:test_case_time_ind} follows completely analogously to the time-dependent case, as ${a}$ is coercive and $f_{\text{st}} \in \cY_{\text{st}}'$. 
	As in the time-dependent case, we choose $V_h$ as linear FE space and $\bar \cYd^{\!x} \subset H^1(\Omega_x)$ as continuous $\mathbb{Q}_2$ FE space on a 2D uniform rectangular mesh.
}

\subsection{Numerical results} 

{
	We first compute the discrete solution to \eqref{eq:test_case_timdep} for $q^{-1}=0.8$ and $h_v=h_{x_1}=h_{x_2}=h_t=1/16$. The assembly of the system matrices which includes the computation of the $\cXd$ basis functions as described in \cref{sec:comp} takes up about 11\% of the computational time in our experiments. Hence, the additional low-dimensional problems in $V_h$ are not dominant in the computational costs. 
	In \cref{fig:plots_FP}, plots of the solution are shown, where we see that the dynamics of the solution are captured well and that no instabilities or oscillations occur. 
}

\begin{figure}
	\vspace{-1em}
	\begin{minipage}{0.27\linewidth}
		%	\vspace{0pt}
		%	
		\adjincludegraphics[width=\linewidth, trim={0.15\width} 0 {0.1\width} {0.1\height} , clip]{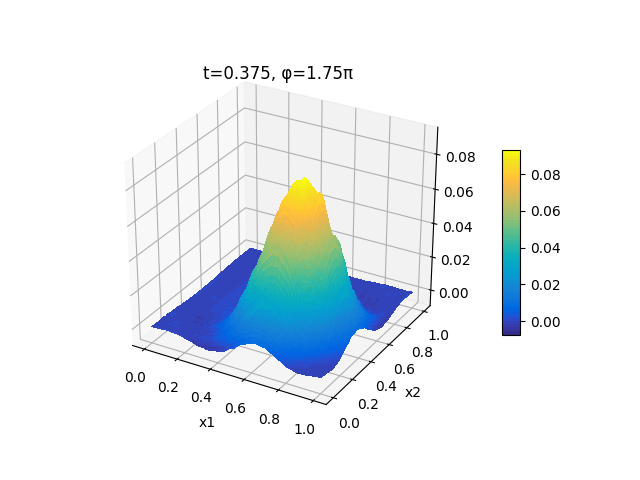}%
		
		\adjincludegraphics[width=\linewidth, trim={0.15\width} 0 {0.1\width} {0.1\height} , clip]{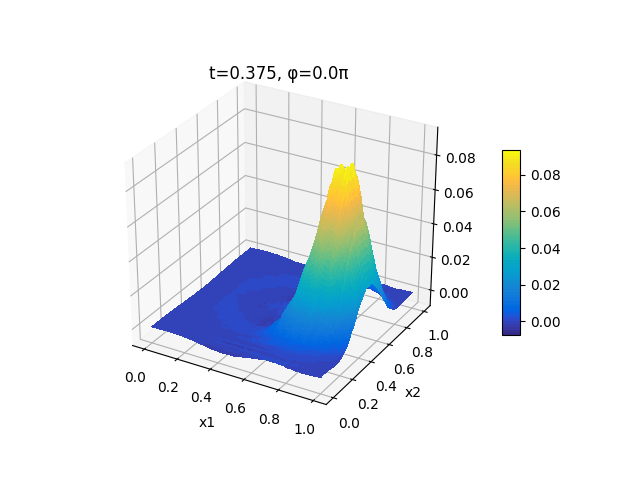}
	\end{minipage}
	\hfill
	\begin{minipage}{0.72\linewidth}%
		
		\vspace{0pt}%
		\hspace{0.09\linewidth}
		\adjincludegraphics[width=0.3\linewidth, trim={0.15\width} 0 {0.25\width} {0.1\height}  , clip]{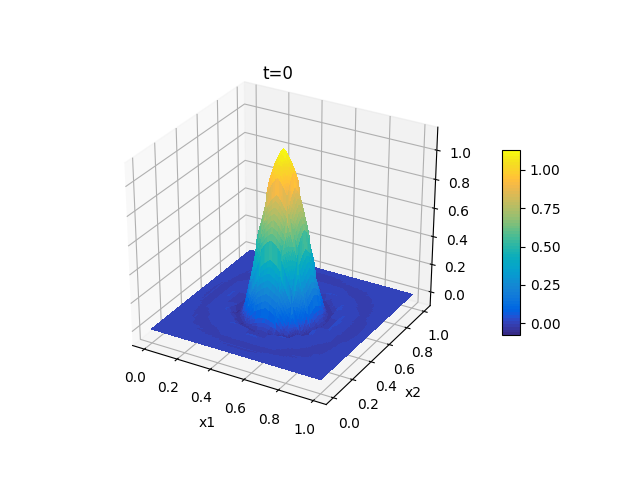}%
		\adjincludegraphics[width=0.3\linewidth, trim={0.15\width} 0 {0.25\width} {0.1\height}  , clip]{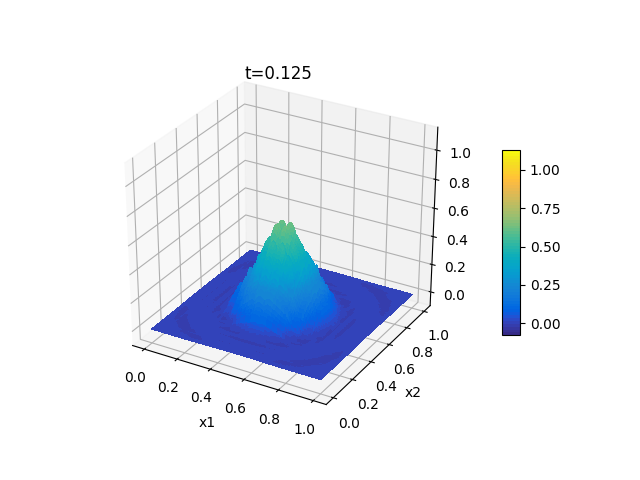}%
		\adjincludegraphics[width=0.3\linewidth, trim={0.15\width} 0 {0.25\width} {0.1\height}  , clip]{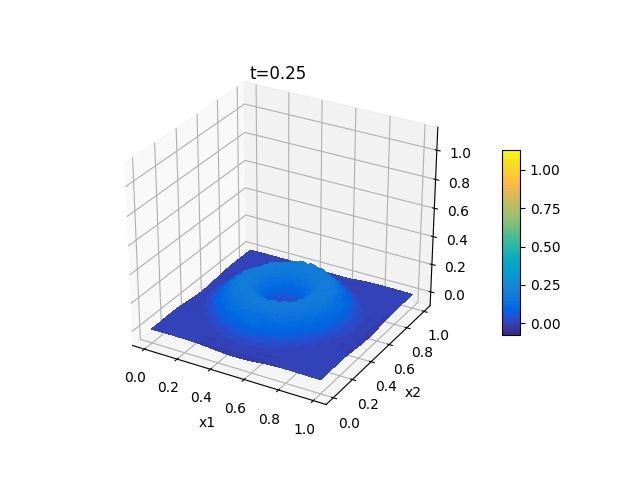}
		
		{
			\hspace{0.21\linewidth}
			\adjincludegraphics[width=0.3\linewidth, trim={0.15\width} 0 {0.25\width} {0.1\height}  , clip]{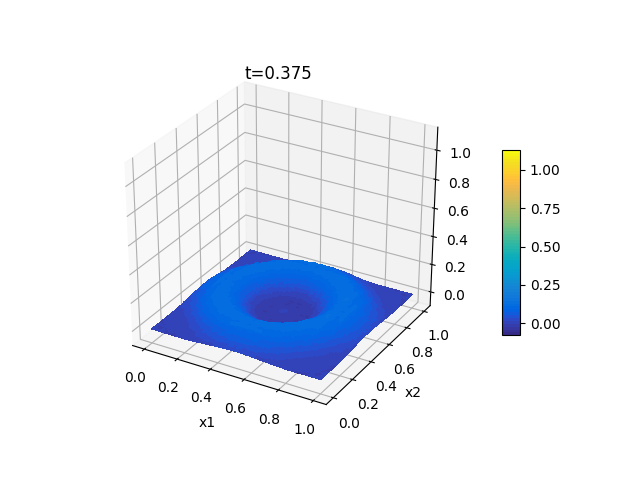}%
			\adjincludegraphics[width=0.383\linewidth, trim={0.15\width} 0 {0.1\width} {0.1\height}  , clip]{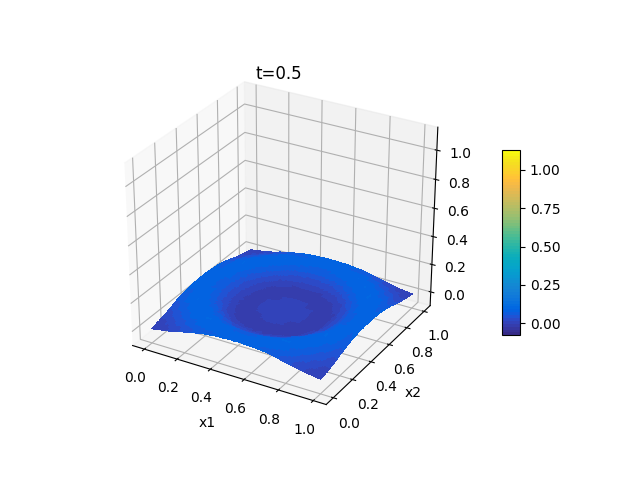}
		}
	\end{minipage}	
	\caption{ Plots of the solution $u^\delta$ of \eqref{eq:test_case_timdep} with $q^{-1} = 0.8$
		for $h_v=h_{x_1}=h_{x_2}=h_t=1/16$. Left: Solution for fixed $t=0.375$ and $\varphi=1.75\pi$ (upper) and $\varphi=0$ (lower). Right: Spatial density, i.e., moment $\int_{0}^{2\pi}u(\cdot,\cdot,\varphi)d\varphi$ for different $t$.
		\label{fig:plots_FP}}
	\vspace{-1em}
\end{figure}
{
	To investigate whether the estimate for the discrete inf-sup constant from \cref{sect:discretization} is sharp, we compute the constants for the transformed problem with $\hat a(\cdot,\cdot)=q^{-1}(\cdot,\cdot)_V$ for different $q^{-1}$ and different mesh sizes. In \cref{table:inf_sup_timdep}, we show the evaluated constants in relation to the lower bound \eqref{eq:discrete_inf_sup}, which is given for this test case as $q^{-1}/(\sqrt{2}\max(1,q^{-1}))$. We see that the estimate is sharp up to a factor of about $\sqrt{2}$.
}

\begin{table}
	\caption{Discretization of \cref{eq:test_case_timdep}: Computed discrete inf-sup constants $\beta_\delta$ of the transformed problem in relation to the respective lower bound $\beta_{\text{lb}}$ for varying mesh sizes with $n=1/{h_{x_1}} = 1/{h_{x_2}} = 1/h_t = 2\pi/h_v$.}
	\label{table:inf_sup_timdep}
	\begin{center}
		{\footnotesize 
			\begin{tabular}{|c|l|l||l|l||l|l|}\hline
				& \multicolumn{2}{c||}{$q^{-1}=0.8$} &  \multicolumn{2}{c||}{$q^{-1}=0.4$} & \multicolumn{2}{c|}{$q^{-1}=0.1$} \\ \cline{2-7}
				$n$ & \multicolumn{1}{c|}{$\beta_\delta$}
				& \multicolumn{1}{c||}{$\beta_\delta/\beta_{\text{lb}}\!$}
				& \multicolumn{1}{c|}{$\beta_\delta$}
				& \multicolumn{1}{c||}{$\beta_\delta/\beta_{\text{lb}}\!$}
				& \multicolumn{1}{c|}{$\beta_\delta$}
				& \multicolumn{1}{c|}{$\beta_\delta/\beta_{\text{lb}}\!$} \\
				\hline 
				4 & 0.8878
				& 1.569
				& 0.6418
				& 2.269
				& 0.45005
				& 6.365
				\\
				\hline 
				8 & 0.81141
				& 1.434
				& 0.44126
				& 1.56
				& 0.18668
				& 2.64
				\\ 
				\hline 
				$\!16\!$ & 0.80072
				& 1.415
				& 0.40317
				& 1.425
				& 0.11112
				& 1.573
				\\
				\hline 
			\end{tabular}
		}
	\end{center}

\end{table}

%%%%%%%%%%%%%%%%%%%%%%%%%%%%%%%%%%%%%%%%%%%%%%%%%%%%%%%%%%%%%%%%%%%%%

{
	To examine the convergence behavior of our scheme, we compute discrete solutions to \eqref{eq:test_case_time_ind}, where the exact solution is known. We compare the discretization errors for different mesh sizes in the $L^2(\Omega_x \times \Omega_v)$ norm as well as in the $\cX_{\text{st}}$ norm in \cref{Figure:Discretization_errors}.
	We see that the $L^2$-error converges with second order in both $h_{x_1}=h_{x_2}=1/n_x$ and $h_v = 2\pi/n_v$. The $\cX_{\text{st}}$-error, which includes the $L^2$-norm of the $v$-derivative, converges with second order in $h_{x_i}$ and first order in $h_v$. 
}
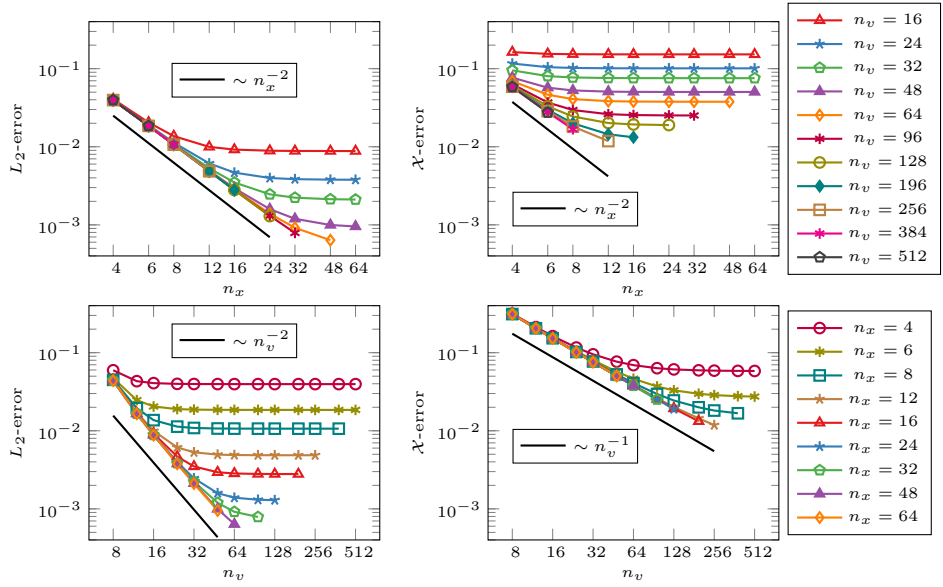
\begin{figure}		
	\tiny{
		\begin{tikzpicture}
		\pgfplotsset{
			log x ticks with fixed point/.style={
				xticklabel={
					\pgfkeys{/pgf/fpu=true}
					\pgfmathparse{exp(\tick)}%
					\pgfmathprintnumber[fixed relative, precision=3]{\pgfmathresult}
					\pgfkeys{/pgf/fpu=false}
				}
			},
			log y ticks with fixed point/.style={
				yticklabel={
					\pgfkeys{/pgf/fpu=true}
					\pgfmathparse{exp(\tick)}%
					\pgfmathprintnumber[fixed relative, precision=3]{\pgfmathresult}
					\pgfkeys{/pgf/fpu=false}
				}
			}
		}
		\definecolor{c1}{RGB}{228,26,28}
		\definecolor{c2}{RGB}{55,126,184}
		\definecolor{c3}{RGB}{77,175,74}
		\definecolor{c4}{RGB}{152,78,163}
		\definecolor{c5}{RGB}{255,127,0}
		\begin{loglogaxis}[width=155pt, legend style={at={(1.03,0.5)}, anchor=west}, xlabel=$n_x$, ylabel=$L_2$-error, xtick={4, 6, 8, 12, 16, 24, 32, 48, 64}, log x ticks with fixed point, ymin=0.0004, ymax=0.4]
		\addplot [color=c1, mark=triangle, thick] table [xlabel={nv}, y index=3, col sep=comma]{L2_errors_transposed.csv};
		%\addlegendentry{$n_v=16$}
		%
		\addplot [color=c2, mark=star, thick] table [xlabel={nv}, y index=4, col sep=comma]{L2_errors_transposed.csv};
		%\addlegendentry{$n_v=24$}
		%
		\addplot [color=c3, mark=pentagon, thick] table [xlabel={nv}, y index=5, col sep=comma]{L2_errors_transposed.csv};
		%\addlegendentry{$n_v=32$}
		%
		\addplot [color=c4, mark=triangle*, thick] table [xlabel={nv}, y index=6, col sep=comma]{L2_errors_transposed.csv};
		%\addlegendentry{$n_v=48$}
		%
		\addplot [color=c5, mark=diamond, thick] table [xlabel={nv}, y index=7, col sep=comma]{L2_errors_transposed.csv};
		%\addlegendentry{$n_v=64$}
		%
		\addplot [color=purple, mark=asterisk, thick] table [xlabel={nv}, y index=8, col sep=comma]{L2_errors_transposed.csv};
		%\addlegendentry{$n_v=96$}
		\addplot [color=olive, mark=o, thick] table [xlabel={nv}, y index=9, col sep=comma]{L2_errors_transposed.csv};
		%\addlegendentry{$n_v=128$}
		\addplot [color=teal, mark=diamond*, thick] table [xlabel={nv}, y index=10, col sep=comma]{L2_errors_transposed.csv};
		%\addlegendentry{$n_v=196$}
		\addplot [color=brown, mark=square, thick] table [xlabel={nv}, y index=11, col sep=comma]{L2_errors_transposed.csv};
		%\addlegendentry{$n_v=256$}
		\addplot [color=magenta, mark=asterisk, thick] table [xlabel={nv}, y index=12, col sep=comma]{L2_errors_transposed.csv};
		%\addlegendentry{$n_v=384$}
		\addplot [color=darkgray, mark=pentagon, thick] table [xlabel={nv}, y index=13, col sep=comma]{L2_errors_transposed.csv};
		%\addlegendentry{$n_v=512$}
		\addplot[black, thick, domain=4:24, samples=200,]{0.4*x^(-2)}; \label{ord_2}
		\node [draw,fill=white,font=\tiny] at (rel axis cs: 0.5,0.75) {\shortstack[l]{
				\ref{ord_2} $\sim n_x^{-2}$}};
		%\addlegendentry{$\sim {} n_x^{-2}$}
		\end{loglogaxis}
		\end{tikzpicture}%
		\hspace{1em}	
		\begin{tikzpicture}
		\pgfplotsset{
			log x ticks with fixed point/.style={
				xticklabel={
					\pgfkeys{/pgf/fpu=true}
					\pgfmathparse{exp(\tick)}%
					\pgfmathprintnumber[fixed relative, precision=3]{\pgfmathresult}
					\pgfkeys{/pgf/fpu=false}
				}
			},
			log y ticks with fixed point/.style={
				yticklabel={
					\pgfkeys{/pgf/fpu=true}
					\pgfmathparse{exp(\tick)}%
					\pgfmathprintnumber[fixed relative, precision=3]{\pgfmathresult}
					\pgfkeys{/pgf/fpu=false}
				}
			}
		}
		\definecolor{c1}{RGB}{228,26,28}
		\definecolor{c2}{RGB}{55,126,184}
		\definecolor{c3}{RGB}{77,175,74}
		\definecolor{c4}{RGB}{152,78,163}
		\definecolor{c5}{RGB}{255,127,0}
		\begin{loglogaxis}[width=155pt, legend style={at={(1.03,0.5)}, anchor=west}, xlabel=$n_x$, ylabel=$\cX$-error, xtick={4, 6, 8, 12, 16, 24, 32, 48, 64}, log x ticks with fixed point, ymin=0.0004, ymax=0.4]
		\addplot [color=c1, mark=triangle, thick] table [xlabel={nv}, y index=3, col sep=comma]{X_errors_transposed.csv};
		\addlegendentry{$n_v=16$}
		\addplot [color=c2, mark=star, thick] table [xlabel={nv}, y index=4, col sep=comma]{X_errors_transposed.csv};
		\addlegendentry{$n_v=24$}
		\addplot [color=c3, mark=pentagon, thick] table [xlabel={nv}, y index=5, col sep=comma]{X_errors_transposed.csv};
		\addlegendentry{$n_v=32$}
		\addplot [color=c4, mark=triangle*, thick] table [xlabel={nv}, y index=6, col sep=comma]{X_errors_transposed.csv};
		\addlegendentry{$n_v=48$}
		\addplot [color=c5, mark=diamond, thick] table [xlabel={nv}, y index=7, col sep=comma]{X_errors_transposed.csv};
		\addlegendentry{$n_v=64$}
		\addplot [color=purple, mark=asterisk, thick] table [xlabel={nv}, y index=8, col sep=comma]{X_errors_transposed.csv};
		\addlegendentry{$n_v=96$}
		\addplot [color=olive, mark=o, thick] table [xlabel={nv}, y index=9, col sep=comma]{X_errors_transposed.csv};
		\addlegendentry{$n_v=128$}
		\addplot [color=teal, mark=diamond*, thick] table [xlabel={nv}, y index=10, col sep=comma]{X_errors_transposed.csv};
		\addlegendentry{$n_v=196$}
		\addplot [color=brown, mark=square, thick] table [xlabel={nv}, y index=11, col sep=comma]{X_errors_transposed.csv};
		\addlegendentry{$n_v=256$}
		\addplot [color=magenta, mark=asterisk, thick] table [xlabel={nv}, y index=12, col sep=comma]{X_errors_transposed.csv};
		\addlegendentry{$n_v=384$}
		\addplot [color=darkgray, mark=pentagon, thick] table [xlabel={nv}, y index=13, col sep=comma]{X_errors_transposed.csv};
		\addlegendentry{$n_v=512$}
		\addplot[black, thick, domain=4:12, samples=200,]{0.6*x^(-2)};
		\label{ord_2}
		\node [draw,fill=white,font=\tiny] at (rel axis cs: 0.3,0.2) {\shortstack[l]{
				\ref{ord_2} $\sim n_x^{-2}$}};
		%\addlegendentry{$\sim {} n_x^{-2}$}
		\end{loglogaxis}
		\end{tikzpicture}
		\hfill
	}
	
	\tiny{
		\begin{tikzpicture}
		\pgfplotsset{
			log x ticks with fixed point/.style={
				xticklabel={
					\pgfkeys{/pgf/fpu=true}
					\pgfmathparse{exp(\tick)}%
					\pgfmathprintnumber[fixed relative, precision=3]{\pgfmathresult}
					\pgfkeys{/pgf/fpu=false}
				}
			},
			log y ticks with fixed point/.style={
				yticklabel={
					\pgfkeys{/pgf/fpu=true}
					\pgfmathparse{exp(\tick)}%
					\pgfmathprintnumber[fixed relative, precision=3]{\pgfmathresult}
					\pgfkeys{/pgf/fpu=false}
				}
			}
			%				compat=1.11,
			%				legend image code/.code={
			%					\draw[mark repeat=2,mark phase=2]
			%					plot coordinates {
			%						(0cm,0cm)
			%						(0.15cm,0cm)        %% default is (0.3cm,0cm)
			%						(0.3cm,0cm)         %% default is (0.6cm,0cm)
			%					};%
			%			}
		}
		\definecolor{c1}{RGB}{228,26,28}
		\definecolor{c2}{RGB}{55,126,184}
		\definecolor{c3}{RGB}{77,175,74}
		\definecolor{c4}{RGB}{152,78,163}
		\definecolor{c5}{RGB}{255,127,0}
		\begin{loglogaxis}[width=155pt, legend style={at={(1.03,0.5)}, anchor=west}, xlabel=$n_v$, ylabel=$L_2$-error, xtick={8, 16, 32, 64, 128, 256, 512}, log x ticks with fixed point, ymin=0.0004, ymax=0.4]
		\addplot [color=purple, mark=o, thick] table [xlabel={nv}, y index=1, col sep=comma]{L2_errors_as_table.csv};
		%\addlegendentry{$n_x=4$}
		%
		\addplot [color=olive, mark=asterisk, thick] table [xlabel={nv}, y index=2, col sep=comma]{L2_errors_as_table.csv};
		%\addlegendentry{$n_x=6$}
		%
		\addplot [color=teal, mark=square, thick] table [xlabel={nv}, y index=3, col sep=comma]{L2_errors_as_table.csv};
		%\addlegendentry{$n_x=8$}
		%
		\addplot [color=brown, mark=star, thick] table [xlabel={nv}, y index=4, col sep=comma]{L2_errors_as_table.csv};
		%\addlegendentry{$n_x=12$}
		%
		\addplot [color=c1, mark=triangle, thick] table [xlabel={nv}, y index=5, col sep=comma]{L2_errors_as_table.csv};
		%\addlegendentry{$n_x=16$}
		%
		\addplot [color=c2, mark=star, thick] table [xlabel={nv}, y index=6, col sep=comma]{L2_errors_as_table.csv};
		%\addlegendentry{$n_x=24$}
		%
		\addplot [color=c3, mark=pentagon, thick] table [xlabel={nv}, y index=7, col sep=comma]{L2_errors_as_table.csv};
		%\addlegendentry{$n_x=32$}
		%
		\addplot [color=c4, mark=triangle*, thick] table [xlabel={nv}, y index=8, col sep=comma]{L2_errors_as_table.csv};
		%\addlegendentry{$n_x=48$}
		%
		\addplot [color=c5, mark=diamond, thick] table [xlabel={nv}, y index=9, col sep=comma]{L2_errors_as_table.csv};
		%\addlegendentry{$n_x=64$}
		%
		\addplot[black, thick, domain=8:48, samples=200,]{x^(-2)};\label{ord_2}
		\node [draw,fill=white,font=\tiny] at (rel axis cs: 0.5,0.85) {\shortstack[l]{
				\ref{ord_2} $\sim n_v^{-2}$}};
		\end{loglogaxis}
		\end{tikzpicture}%
		\hspace{1em}
		\begin{tikzpicture}
		\pgfplotsset{
			log x ticks with fixed point/.style={
				xticklabel={
					\pgfkeys{/pgf/fpu=true}
					\pgfmathparse{exp(\tick)}%
					\pgfmathprintnumber[fixed relative, precision=3]{\pgfmathresult}
					\pgfkeys{/pgf/fpu=false}
				}
			},
			log y ticks with fixed point/.style={
				yticklabel={
					\pgfkeys{/pgf/fpu=true}
					\pgfmathparse{exp(\tick)}%
					\pgfmathprintnumber[fixed relative, precision=3]{\pgfmathresult}
					\pgfkeys{/pgf/fpu=false}
				}
			}
		}
		\definecolor{c1}{RGB}{228,26,28}
		\definecolor{c2}{RGB}{55,126,184}
		\definecolor{c3}{RGB}{77,175,74}
		\definecolor{c4}{RGB}{152,78,163}
		\definecolor{c5}{RGB}{255,127,0}
		\begin{loglogaxis}[width=155pt, legend style={at={(1.03,0.5)}, anchor=west}, xlabel=$n_v$, ylabel=$\cX$-error, xtick={8, 16, 32, 64, 128, 256, 512}, log x ticks with fixed point, ymin=0.0004, ymax=0.4]
		\addplot [color=purple, mark=o, thick] table [xlabel={nv}, y index=1, col sep=comma]{X_errors_as_table.csv};
		\addlegendentry{$n_x=4$}
		\addplot [color=olive, mark=asterisk, thick] table [xlabel={nv}, y index=2, col sep=comma]{X_errors_as_table.csv};
		\addlegendentry{$n_x=6$}
		\addplot [color=teal, mark=square, thick] table [xlabel={nv}, y index=3, col sep=comma]{X_errors_as_table.csv};
		\addlegendentry{$n_x=8$}
		\addplot [color=brown, mark=star, thick] table [xlabel={nv}, y index=4, col sep=comma]{X_errors_as_table.csv};
		\addlegendentry{$n_x=12$}
		\addplot [color=c1, mark=triangle, thick] table [xlabel={nv}, y index=5, col sep=comma]{X_errors_as_table.csv};
		\addlegendentry{$n_x=16$}
		\addplot [color=c2, mark=star, thick] table [xlabel={nv}, y index=6, col sep=comma]{X_errors_as_table.csv};
		\addlegendentry{$n_x=24$}
		\addplot [color=c3, mark=pentagon, thick] table [xlabel={nv}, y index=7, col sep=comma]{X_errors_as_table.csv};
		\addlegendentry{$n_x=32$}
		\addplot [color=c4, mark=triangle*, thick] table [xlabel={nv}, y index=8, col sep=comma]{X_errors_as_table.csv};
		\addlegendentry{$n_x=48$}
		\addplot [color=c5, mark=diamond, thick] table [xlabel={nv}, y index=9, col sep=comma]{X_errors_as_table.csv};
		\addlegendentry{$n_x=64$}
		\addplot[black, thick, domain=8:256, samples=200,]{1.4*x^(-1)}; \label{ord_1}
		\node [draw,fill=white,font=\tiny] at (rel axis cs: 0.3,0.4) {\shortstack[l]{
				\ref{ord_1} $\sim n_v^{-1}$}};
		%\addlegendentry{$\sim {} {n_v}^{-1}$}
		\end{loglogaxis}
		\end{tikzpicture}%
		\hfill
	}		
	\caption{Discretization of \eqref{eq:test_case_time_ind} with $d=0.1$, $c=0.1$.
		Left: $L^2$-errors $\|u - u_\delta\|_{L^2(\Omega_{x}\times\Omega_v)}$, right: $\cX_{\mathrm{st}}$-errors $\|u - u_\delta\|_{L^2(\Omega_{x}, V)}$. Upper plots: convergence in $n_x = 1/h_{x_1}=1/h_{x_2}$ for different fixed $n_v = 2\pi/h_v$. Lower plots: Convergence in $n_v$ for different fixed $n_x$.}
	\label{Figure:Discretization_errors}
	\vspace{-1em}
\end{figure}
{
	For a further investigation of the estimate for the discrete inf-sup constant 
	we compute the constants for the discretization of \eqref{eq:test_case_time_ind} for different mesh sizes and reaction and diffusion constants $c$ and $d$; see \cref{table:discrete_inf_sup}. The estimate {\eqref{eq:discrete_inf_sup}} is given here as $\beta_\delta \geq \min\{c,d\} /(\sqrt{2}\max\{1, c, d\})$, which is $\min\{c,d\}/\sqrt{2}$ for all considered data values in \cref{table:discrete_inf_sup}. As can be seen in the table, the estimate is here again sharp up to a factor of about $\sqrt{2}$.
}
\begin{table}[tb]	
	\begin{minipage}{0.69\textwidth} \centering
		\captionsetup{width=0.95\linewidth}
		\captionof{table}{Discretization of \cref{eq:test_case_time_ind}: Computed discrete inf-sup constants $\beta_\delta$ in relation to the lower bound $\beta_{\text{lb}}$ for varying mesh sizes with $n=1/{h_{x_1}} = 1/{h_{x_2}} = 2\pi/h_v$ and varying values for the constants $d$ and $c$.}\label{table:discrete_inf_sup}
		{\footnotesize
			\begin{tabular}{|c|l|l||l|l||l|l|}\hline
				& \multicolumn{2}{c||}{$d=0.4$, $c=1$} &  \multicolumn{2}{c||}{$d=0.1$, $c=1$} & \multicolumn{2}{c|}{$d=0.1$, $c=0.1$} \\ \cline{2-7}
				$n$ & \multicolumn{1}{c|}{$\beta_\delta$}
				& \multicolumn{1}{c||}{$\beta_\delta/\beta_{\text{lb}}\!$}
				& \multicolumn{1}{c|}{$\beta_\delta$}
				& \multicolumn{1}{c||}{$\beta_\delta/\beta_{\text{lb}}\!$}
				& \multicolumn{1}{c|}{$\beta_\delta$}
				& \multicolumn{1}{c|}{$\beta_\delta/\beta_{\text{lb}}\!$} \\
				\hline 
				4 & 0.61855& 2.187
				& 0.41087& 5.811
				& 0.30579& 4.324
				\\
				\hline 
				8 & 0.44891 & 1.587
				& 0.18628 & 2.634
				& 0.14924& 2.111
				\\ 
				\hline 
				$\!16\!$ & 0.40915 & 1.447
				& 0.11688 & 1.653
				&0.10585 & 1.497
				\\
				\hline 
				$\!32\!$ & 0.40202 & 1.421
				& 0.1033 & 1.461
				& 0.10041 & 1.42
				\\ 
				\hline 
				$\!48\!$ & 0.40088 & 1.417
				& 0.10137 & 1.434
				&0.10008 & 1.415
				\\ 
				\hline 
			\end{tabular}
		}	
	\end{minipage}
	\hfill
	\begin{minipage}{0.3\textwidth}
		\captionsetup{width=.99\linewidth}
		\captionof{table}{Ratio of nonzero elements in the stiffness matrix for varying mesh sizes}
		\label{table:sparsity}
		{\footnotesize
			\begin{tabular}{|c|c|c|}
				\hline
				$n$ & $\frac{n_{\text{nz}}}{n_{\text{entries}}}$ & $\!\!\frac{n_{\text{nz}}}{(n_{x_1} n_{x_2} n_v^2)}\!\!$ \\ \hline
				4 &	20.05\%	& 39.3 \\ \hline
				8 &	5.52\%	& 53.05 \\ \hline
				16 &	1.463\%	& 58.98\\ \hline
				32 &	0.378\%	& 61.61 \\ \hline
				48 &	0.17\%	& 62.44 \\ \hline
				64 &	0.096\%	& 62.84 \\ \hline
			\end{tabular}
		}
	\end{minipage}
\end{table}

Since the basis functions of the discrete trial space $\cXd$ are not chosen as standard nodal basis functions but have larger support, one can ask if the choice of spaces still leads to an efficient numerical scheme. { Therefore, in \cref{table:sparsity} we list the ratio of nonzero elements in the stiffness matrix, which decreases significantly with larger problem sizes. However, as $\cXd$ includes solutions of problems in $\Omega_v$, the nonzero elements increase linearly in the dimension of the $x$-discretization and quadratically in the dimension of $V_h$.
}

\section{Conclusions}

In this paper, we present a stable Petrov-Galerkin discretization of a kinetic Fokker-Planck equation. 
Based on an estimate for the dual inf-sup constant of the bilinear form, where ``stable pairs'' of trial and test functions are introduced, we propose a discretization where these pairs are directly built into the spaces: 
By defining the discrete trial space dependent on the chosen discrete test space through the application of the kinetic transport and the inverse velocity Laplace-Beltrami operator, we obtain a well-posed numerical scheme with the same lower bound of the discrete inf-sup constant as for the continuous problem independently of the mesh size.  
We show that under suitable conditions on the data functions these spaces can be computed efficiently. 
Numerical experiments show { favorable convergence orders of the discretization error for a manufactured solution of the stationary equation (order 2 in $x$ both in the $L^2$-norm and the $\cX$-norm, order 2 and 1 in $v$ for the respective norms). For both the examined time-dependent and stationary test cases, the estimate of the discrete inf-sup constant is sharp up to a factor of $\sqrt{2}$. }

The new method is especially beneficial for spaces with few degrees of freedom in the velocity domain. 
Therefore, a promising application might be a combination with a hierarchical model order reduction scheme such as \cite{BLOS16}, which realizes small spaces in the velocity domain and has stability problems that might be resolved using the new method.

\appendix

\section{Proofs of function space results} \label{sect:app_proofs}

\begin{proof}[Proof of \cref{lem:product_continuous}]
	We estimate $\|\phi f\|_{H^1_{\mathrm{FP}}(\Omega)}$. Using the definition of the $V$-norm  and the product rule we obtain for the first term\footnote{As introduced in \cref{sect:var_form}, we write $\cX = L^2(\Omega_{t,x},V)$.}
	\begin{align} \label{eq:est_fphi_norm_X}
	\|\phi f\|_{\cX}^2 
	&= 
	\|\phi f\|_{L^2(\Omega)}^2
	+ \|(\nabla_v \phi )f +  \phi \nabla_v f\|_{L^2(\Omega)}^2 \nonumber\\
	&\leq 
	\|\phi^2\|_{L^\infty(\Omega)} \|f\|_{L^2(\Omega)}^2
	+ 2\||\nabla_v \phi|^2\|_{L^\infty(\Omega)} \|f\|_{L^2(\Omega)}^2
	+ 2\|\phi^2\|_{L^\infty(\Omega)} \|\nabla_v f\|_{L^2(\Omega)}^2 \nonumber\\
	&\leq 2\left ( \|\phi\|_{L^\infty(\Omega)}^2 + \|\nabla_v \phi\|_{L^\infty(\Omega)}^2 \right ) \|f\|_{\cX}^2.
	\end{align}
	By using the product rule, the {identification} $\la \cdot,\cdot\ra_{\cX'\!,\cX} = (\cdot,\cdot)_{L^2(\Omega)}$, and the {density} of $C^\infty(\Omega)$ in $H^1_{\mathrm{FP}}(\Omega)$ we see that for arbitrary $\psi \in \cX$ it holds
	\begin{align*}
	\la {k} \cdot \nabla_{t,x} &(\phi f), \psi \ra_{\cX',\cX} 
	= \la {k} \cdot \nabla_{t,x} f, \phi  \psi \ra_{\cX',\cX}
	+
	( f ({k} \cdot \nabla_{t,x} \phi ) ,\psi )_{L^2(\Omega)} \\
	&\leq
	\|{k} \cdot \nabla_{t,x} f \|_{\cX'} \|\phi  \psi \|_{\cX} 
	+  \|f ({k} \cdot \nabla_{t,x} \phi ) \|_{L^2(\Omega)} \| \psi\|_{L^2(\Omega)}. \\
	&
	\hspace{-1em}\overset{\cref{eq:est_fphi_norm_X} }{\leq} 
	\sqrt{2} \left  ( \|\phi\|_{L^\infty(\Omega)}^2 + \|\nabla_v \phi \|_{L^\infty(\Omega)}^2 \right )^{\frac{1}{2}}
	\|{k} \cdot \nabla_{t,x} f \|_{\cX'} \|\psi \|_{\cX} \\
	&\quad +
	\|{k} \cdot \nabla_{t,x} \phi\|_{L^\infty(\Omega)} 
	\|f \|_{L^2(\Omega)} 
	\|\psi\|_{L^2(\Omega)}\\
	&\leq \sqrt{2} \left (\|\phi \|_{L^\infty(\Omega)} + \|\nabla_v \phi \|_{L^\infty(\Omega)} + \|{k} \cdot \nabla_{t,x} \phi \|_{L^\infty(\Omega)} \right )
	\|f\|_{H^1_{\mathrm{FP}}(\Omega)} \|\psi\|_{\cX}.
	\end{align*}
	We thus have
	\begin{equation}\label{eq:est_dtx_fphi_X'}
	\begin{split}
	\|{k} \cdot \nabla_{t,x} (\phi f)\|_{\cX'} 
	\leq  2\sqrt{2}
	\left (
	\|\phi\|_{L^\infty(\Omega)} + \|\nabla_v \phi \|_{L^\infty(\Omega)} 
	\right . \;&\\
	\quad\left . + \|{k} \cdot \nabla_{t,x} \phi \|_{L^\infty(\Omega)} 
	\right )&
	\|f\|_{H^1_{\mathrm{FP}}(\Omega)}.
	\end{split}
	\end{equation}
	Combining \eqref{eq:est_fphi_norm_X} and \eqref{eq:est_dtx_fphi_X'} and using that $\left |{k} \right |$ is bounded in $\Omega$, we thus have
	\begin{align*}
	\|\phi f\|_{H^1_{\mathrm{FP}}(\Omega)}
	\leq 
	C \|\phi\|_{C^1(\Omega)} \|f\|_{H^1_{\mathrm{FP}}(\Omega)}. 
	\end{align*}
\end{proof}

\section*{Acknowledgments}
We would like to thank Dr.\ M.\ Schlottbom (University of Twente) and Prof.\ M.\ Ohlberger (University of Münster) for fruitful discussions.

\bibliographystyle{siamplain}

\end{document}